\documentclass[psamsfont,a4paper,12pt]{amsart}

\usepackage[english]{babel}   

\usepackage{setspace}     
\usepackage{enumerate} 	
\usepackage{afterpage} 	
\usepackage[usenames,dvipsnames]{color}
\usepackage{graphicx,hyperref}	

\usepackage{amsmath} 	
\usepackage{amssymb}	
\usepackage{mathrsfs} 	
\usepackage{amsfonts} 	
\usepackage{latexsym} 	
\usepackage[latin1]{inputenc} 
\usepackage{mathtools}

\usepackage{amsthm} 	
\usepackage{stackrel}       
\usepackage{amscd} 	
\usepackage{ dsfont }               

\usepackage{tabularx}	
\usepackage{longtable}	
\usepackage{array}
\usepackage{float}

\usepackage{verbatim} 
\usepackage{blindtext} 
\hypersetup{colorlinks=true,linkcolor=blue,citecolor=magenta}
\usepackage{lipsum}

\allowdisplaybreaks

\def\a{\mathfrak{a}}

\def\m{\mathfrak{m}}

\newcommand{\C}{\mathbb{C}}

\newcommand{\R}{\mathbb{R}}

\newcommand{\Z}{\mathbb{Z}} 
\newcommand{\Q}{\mathbb{Q}}

\renewcommand{\to}{\longrightarrow}

\newtheorem{Thm}{Theorem}[section]
\newtheorem*{theorem*}{Theorem}
\newtheorem*{thm*}{Theorem}
\newtheorem{Lemma}[Thm]{Lemma} 
 
\newtheorem{Prop}[Thm]{Proposition} 
\newtheorem{Cor}[Thm]{Corollary}

\theoremstyle{definition}
\newtheorem{Definition}[Thm]{Definition}
\newtheorem*{defn}{Definition}  

\newtheorem*{Ex}{Example}  

\theoremstyle{remark}
 
\newtheorem*{rmk}{Remark}

\newtheorem{ind}[]{{\rm\it Indice}}

\renewcommand{\epsilon}{\varepsilon}

\usepackage{multicol}

\usepackage{tikz}
\usetikzlibrary{arrows}

\usepackage{paralist}
\usepackage{cite}

\title[Jacobi forms with CM]{Jacobi forms with CM and applications}
\author[Wagner]{Ian Wagner}
\email{ianwagner11@gmail.com}

\usepackage{listings}

\begin{document}
\numberwithin{equation}{section}

\maketitle

\begin{abstract}
We define Jacobi forms with complex multiplication.  Analogous to modular forms with complex multiplication, they are constructed from Hecke characters of the associated imaginary quadratic field.  From this construction we obtain a Jacobi form which specializes to $\eta(\tau)^{26}$ which we present to highlight an open question of Dyson and Serre.  We give other examples and applications of Jacobi forms with complex multiplication including constructing theta blocks associated to elliptic curves with complex multiplication and new families of congruences and cranks for certain partition functions.
\end{abstract}

\section{introduction and statement of results}

Modular forms with complex multiplication (CM) were first introduced by Hecke when he studied Hecke characters in order to simultaneously generalize both Dirichlet $L$-functions and Dedekind zeta functions (see \cite{CS}).  Since that time they have been widely studied in connection to varied subjects such as Apery-like numbers and lacunary modular forms \cite{GMY, OR}.  One especially important work is that of Serre \cite{Serre} where he completely classifies the powers of the Dedekind eta function $\eta(\tau)$ which have CM.  Beautiful formulas for the coefficients of these functions and related ones were first noticed by Dyson.  For example, if we define the coefficients $\tau(n)$ by
\begin{equation*}
 \sum_{n \geq 1} \tau(n) q^n \coloneqq \eta^{24}(\tau) 
\end{equation*}
where we let $q\coloneqq e^{2 \pi i \tau}$ with $ \tau \in \mathbb{H}$ throughout, then a related identity is
\begin{equation*}
\tau(n) =\sum_{\substack{x \in \Z^5 \\ x \cdot 1 =0 \\ x_{i} \equiv i \pmod{5} \\ x_{1}^2 + x_{2}^2 + x_{3}^2 + x_{4}^2 + x_{5}^2 =10n}}\left( \frac{1}{288} \prod_{1 \leq i < j \leq 5} (x_{i}-x_{j}) \right).
\end{equation*}  
This identity and almost all of the others noticed by Dyson were explained by Macdonald \cite{M} in his work which connected these modular forms to affine root systems.  Specifically he constructed Jacobi forms with lattice index of the associated root lattice which, when specialized, simplified to the corresponding modular form. In particular, the identity above comes from the root system $A_{4}$.  The only similar identity that was not explained by Macdonald's work was the one corresponding to $\eta(\tau)^{26}$ which also belongs to Serre's list of powers of $\eta(\tau)$ with CM.  Both Serre \cite{Serre} and Dyson \cite{Dy} asked for a MacDonald-type identity for $\eta(\tau)^{26}$.  With this as motivation we define a family of Jacobi forms which we call {\it Jacobi forms with complex multiplication} as they specialize to modular forms with CM.  We will present several applications and examples of this construction.  In particular, we reproduce the MacDonald identities associated to powers of $\eta(\tau)$ with CM, we construct Jacobi forms with infinite product representations associated to elliptic curves with CM, and we show how to construct crank functions which explain Ramanujan-type congruences for different kinds of partition functions.  All of these applications are instances where Jacobi forms have infinite product representations and so they fit within the recent theory of theta blocks \cite{GSZ} which are essentially quotients of Jacobi theta functions and Dedekind eta functions.

We delay details on Hecke characters until Section 2.  If $K$ is an imaginary quadratic field and $\xi$ is a Hecke character modulo $\m$ with infinity type $\xi_{\infty}(\alpha) = \left(\frac{\alpha}{|\alpha|} \right)^{k-1}$ with $k \geq 1$, then we define
\begin{equation*}
\theta_{\xi}(\tau) \coloneqq \sum_{\mathfrak{a} \ \text{integral}} \xi(\mathfrak{a}) N(\a)^{\frac{k-1}{2}} q^{N(\a)},
\end{equation*}
where Let $N(\mathfrak{a})$ is the norm of the ideal $\mathfrak{a}$.  It is known that $\theta_{\xi}(\tau)$ is a weight $k$ modular form on $\Gamma_{0}(|D| \cdot N(\m))$ with character $\chi(n) = \chi_{D}(n) \xi_{m}(n)$ where $\chi_{D}(n) = \left(\frac{D}{n} \right)$ \cite{CS}.  Further, if $k \geq 2$ then $\theta_{\xi}$ is a cusp form and if $\xi$ is primitive then $\theta_{\xi}$ is a newform.

Throughout we will let $z=a \tau + b \in \C^2$ with $a, b \in \R^2$.  We also delay the definition of $\beta$ and how to define shifts in the arguments of $N$ and $\xi$ until \eqref{idealfunctions} in Section 2, however $\beta$ should be thought of as the bilinear form associated to the quadratic norm form for an imaginary quadratic field.
\begin{Definition}
Assume the notations above.  Let $\xi$ be a Hecke character for the imaginary quadratic field $K=\Q(\sqrt{D})$ modulo $\m$ with infinity type $\xi_{\infty}(\alpha) = \left(\frac{\alpha}{|\alpha|}\right)^{k-1}$. We define the {\bf{Jacobi form with CM by $\xi$}} by
\begin{equation}
\theta_{\xi}(z; \tau) \coloneqq \sum_{\a \ \text{integral}} \xi(\a +a) N(\a + a)^{\frac{k-1}{2}} e^{2 \pi i \beta(\a, b)} q^{N(\a) + \beta(\a, a)}.
\end{equation}
\end{Definition}
We can now state our main theorem. 
\begin{Thm} \label{Main}
Assume the notation above.  Then $\theta_{\xi}(z; \tau)$ is a weight $k$ Jacobi form on $\Gamma_{0}(|D|N(\m))$ with character $\chi_{D} \xi_{\m}$.
\end{Thm}
Given $\theta_{\xi}(z; \tau)$ one can naturally split the sum over all integral ideals of $K$ into a double sum over classes in the ideal class group of $K$, ${\rm Cl}(K)$, and over elements in a representative of the class.  In fact, this is how the proof of Theorem \ref{Main} will proceed.  If one restricts to a single class of the ideal class group, then we still obtain a Jacobi form.  Namely, for $C = [\a_{0}] \in Cl(K)$ define
\begin{equation}
\phi_{\xi, C}(z; \tau) \coloneqq \sum_{\alpha \in \a_0} \xi_{\m}(\alpha) (\alpha + a)^{k-1} e^{2 \pi i B_{C}(\alpha, z)} q^{Q_{C}(\alpha)}.
\end{equation}
Then the proof of Theorem \ref{Main} will show $\phi_{\xi, C}(z; \tau)$ is a weight $k$ Jacobi form on $\Gamma_{0}(|D|N(\m))$ with character $\chi_{D} \xi_{\m}$ and of lattice index $\left(\Z^2, Q_{\a_{0}} \right)$.  The case of $k=1$ is stated in Theorem \ref{P} in Section 2 and has an extra Fricke-like transformation.  In the case where $\Q(\sqrt{D})$ has class number $1$ we have $\theta_{\xi} = \phi_{\xi,C}$.  
This is the situation which is applicable to the powers of $\eta(\tau)$ with CM.  In particular we have the following result related to $\eta(\tau)^{26}$.  Throughout we let $D_{z} \coloneqq \frac{1}{2 \pi i} \frac{\partial}{\partial z}$, $\zeta \coloneqq e^{2 \pi i z}$, and use the standard Jacobi theta functions which are given in Propositions \ref{BB} and \ref{BB2}.
\begin{Thm} \label{26Thm}
Define
\begin{align*}
\phi_{26}(z; \tau) &\coloneqq \frac{1}{2} \sum_{n_1, n_2 \in \Z} (-1)^{\frac{n_1 + n_2 -1}{2}}  \left( \frac{4}{n_1(n_2 +1)} \right) \left(\frac{n_1}{3} \right) \zeta_{1}^{\frac{n_1}{12}} \zeta_{2}^{\frac{n_2}{4}} q^{\frac{n_{1}^2 + 3n_{2}^2}{12}} \\
&+ \frac{1}{4} \sum_{n_{1} \equiv \pm n_2 \pmod{8}} \left(\frac{12}{n_1} \right)  \zeta_{1}^{\frac{n_1}{24}} \zeta_{2}^{\frac{n_2}{8}} q^{\frac{n_{1}^2 + 3n_{2}^2}{48}} \\
&-\frac{1}{2} \sum_{(n_1, n_2) \equiv (\pm 1, 0) \pmod{3}} (-1)^{n_2} \left( \frac{4}{n_1 +n_2} \right) \zeta_{1}^{\frac{n_{1}}{6}} \zeta_{2}^{\frac{n_{2}}{6}} q^{\frac{n_{1}^2 + n_{2}^2}{12}} \\
&+\frac{1}{2} \sum_{(n_1, n_2) \equiv (0, \pm 1) \pmod{3}} (-1)^{n_2} \left( \frac{4}{n_1 +n_2} \right) \zeta_{1}^{\frac{n_{1}}{6}} \zeta_{2}^{\frac{n_{2}}{6}} q^{\frac{n_{1}^2 + n_{2}^2}{12}}.
\end{align*}
Then 
\begin{align*}
\phi_{26}(z; \tau) &=\frac{1}{2} \theta^{*} \left(\frac{z_3}{6}; 2 \tau \right) \theta_{4} \left( \frac{z_4}{2}; 2 \tau \right) + \frac{1}{2} \theta^{*} \left(\frac{z_3}{12}; \frac{\tau}{2} \right) \theta_{2} \left( \frac{z_4}{4}; \frac{\tau}{2} \right) \\
&+ \frac{1}{4} \left[ \theta^{*}\left( \frac{z_3}{12} +\frac{1}{4}; \frac{\tau}{2} \right) \theta_{2} \left(\frac{z_4}{4} -\frac{1}{4}; \frac{\tau}{2} \right) + \theta_{2}^{*} \left(\frac{z_3}{12} +\frac{1}{4}; \frac{\tau}{2} \right) \theta \left(\frac{z_4}{4} -\frac{1}{4}; \frac{\tau}{2} \right) \right] \\
&+ \frac{1}{4} \left[ \theta^{*}\left( \frac{z_3}{12} +\frac{1}{4}; \frac{\tau}{2} \right) \theta_{2} \left(\frac{z_4}{4} +\frac{1}{4}; \frac{\tau}{2} \right) - \theta_{2}^{*} \left(\frac{z_3}{12} +\frac{1}{4}; \frac{\tau}{2} \right) \theta \left(\frac{z_4}{4} +\frac{1}{4}; \frac{\tau}{2} \right) \right] \\
&-\frac{1}{2} \theta^{*} \left(\frac{z_1 + z_2}{6}; \tau \right) \theta^{*} \left(\frac{z_1 -z_2}{6}; \tau \right),
\end{align*}
where $z= z_1 +iz_2 = \frac{z_3 +\sqrt{-3}z_4}{2}$ for $z_j \in \R$.  Further,
\begin{equation*}
D_{z}^{12} \left[\phi_{26}(z; \tau)\right]_{z=0} = \frac{1}{2^8 \cdot 3^4 \cdot 11^2 \cdot 13} \eta(\tau)^{26}.
\end{equation*}
\end{Thm}
\begin{rmk}
To the author's knowledge this is the only known example of a Jacobi form that specializes to $\eta(\tau)^{26}$ which is a sum of infinite products and has a sum-side representation in terms of binary quadratic forms.  However, to truly answer Dyson and Serre's question one would hope for infinite products similar to the other examples in Section 4.1.  In particular, one expects a product of $12$ theta functions divided by $\eta(\tau)^{10}$.
\end{rmk}

The problem of Dyson and Serre was the main motivation for this work, but we now give two further applications of Jacobi forms with CM.  The first application is to construct theta blocks associated to ellitptic curves with CM given in Table 1.  The full statement is given in Theorem \ref{elliptic} in Section 4, but we highlight one in particular here.  This example is of particular interest as it appears related to the cubic theta functions studied in \cite{Bor}.  We also note that the theta blocks associated to elliptic curves given in Theorem \ref{elliptic} have first Taylor coefficients whose square when evaluated at certain CM points are essentially the central $L$-value of the $L$-function of the elliptic curve as shown in \cite{RV}.  When we say a weight $2$ newform corresponds to an elliptic curve it is meant via the modularity theorem of Taylor and Wiles \cite{W, TW}.
\begin{Thm} \label{Ell27}
Let $E_{27}$ denote the elliptic curve
\begin{equation*}
E_{27}: y^2 +y = x^3 -7
\end{equation*}
which has conductor $27$ and CM by $\Q(\sqrt{-3})$ and let $\rho = e^{\frac{\pi i}{3}}$.  Then
\begin{align*}
&\phi_{27}(z; 3\tau) = \sum_{n_1 \equiv n_2 \pmod{2}} \left( \frac{2n_1}{3} \right) \rho^{3n_1 -n_2} \zeta_{1}^{\frac{n_{1}}{6}} \zeta_{2}^{\frac{n_{2}}{2}} q^{\frac{3\left(n_{1}^2 +3n_{2}^2\right)}{4}} \\
&= \frac{\theta\left(\frac{z_1 +3z_2}{6} + \frac{1}{3};3 \tau \right) \theta \left( \frac{z_1 -3z_2}{6} +\frac{2}{3}; 3\tau \right) \theta \left(\frac{z_1}{3};3 \tau \right)}{\eta(3\tau)}
\end{align*}
is a weight $1$ Jacobi form such that $D_{z}\left[\phi_{27}(z;\tau) \right]_{z=0}$ is (up to a nonzero constant) the weight $2$ newform corresponding to $E_{27}$.
\end{Thm}

The last application is combinatorial in nature.  A partition of a non-negative integer $n$ is a non-increasing sequence of positive integers that sum to $n$.  The partition function $p(n)$ counts the number of partitions of $n$ and has generating function
\begin{equation*}
\sum_{n \geq 0} p(n) q^n = \prod_{n \geq 1} \frac{1}{1-q^n} = q^{\frac{1}{24}} \eta(\tau)^{-1}.
\end{equation*}
 Ramanujan \cite{Ramanujan} famously gave three congruences modulo $5, 7$, and $11$ for the partition function:
 \begin{align*}
 p(5n+4) & \equiv 0 \pmod{5} \\
 p(7n+5) &\equiv 0 \pmod{7} \\
 p(11n+6) & \equiv 0 \pmod{11}.
 \end{align*}
Many generating function proofs of these congruences exist, but Dyson famously asked for a combinatorial explanation for these congruences.  Specifically he asked for a statistic on partitions that splits the number of partitions of $\ell n + \delta$ into $\ell$ equinumerous sets of size $\frac{1}{\ell}p(\ell n + \delta)$ for $(\ell, \delta)=(5,4), (7,5), (11,6)$.  Dyson defined the {\it rank} of a partition as the largest part of the partition minus the number of parts and in \cite{AtkinSD} Atkin and Swinnerton-Dyer showed that this statistic achieved Dyson's goal for $\ell=5$ and $7$.  However the rank does not explain the Ramanujan congruence modulo $11$ and so Dyson conjectured the existence of another statistic, called the {\it crank}, that would explain all three congruences simultaneously.  Garvan discovered a variant of a crank function in \cite{Garvan1, Garvan2} and Andrews and Garvan later constructed Dyson's crank in \cite{AndrewsGarvan}.  In Section 4 we will show how to simultaneously prove Ramanujan-type congruences and construct crank functions which define statistics which explain the congruences for a family of partition functions.  This idea was recently used to construct crank functions for $k$-colored partitions in \cite{RTW} and \cite{BGRT}.  A general theorem of this nature is given in Theorem \ref{Cong}, but as an example of this result we present the following.  Define a {\it $(3,2)$-colored overpartition} as a partition with $3$ colors such that the first occurrence of a number may be overlined in $2$ of the colors.  We denote the number of such partitions of $n$ by $\overline{p}_{3,2}(n)$.  We then have the following result.
\begin{Thm} \label{CC}
For all integers $n \geq 0$ we have the Ramanujan-type congruence
\begin{equation*}
\overline{p}_{3,2}(7n+5) \equiv 0 \pmod{7}.
\end{equation*}
Further,
\begin{equation*}
\overline{C}_{3,2}(z; \tau) \coloneqq \prod_{n \geq 1} \frac{(1-q^n)(1+q^n)^2}{(1-\zeta q^n)(1-\zeta^{-1}q^n)(1-\zeta^2 q^n)(1-\zeta^{-2}q^n)}
\end{equation*}
defines a crank function which explains the congruence.
\end{Thm}

The paper is organized as follows.  Section 2 will include a brief background on Hecke characters and present a specific case of Jacobi forms with CM.  We also state a theorem of Vign\'{e}ras which is the main ingredient needed to prove Theorem \ref{Main} and Theorem \ref{P} in Section 3.  Section 4 will include all examples, applications, and the proofs of Theorems \ref{26Thm}, \ref{Ell27}, and \ref{CC}.  Section 4 begins with the Macdonald identities for powers of $\eta(\tau)$ and the proof of Theorem \ref{26Thm} in Subsection 4.1.  Subsection 4.2 will present the theta blocks associated to elliptic curves with CM and Subsection 4.3 includes a new family of partition functions, Ramanujan-type congruences for them, and examples of how to construct crank functions which explain those congruences.

\section*{Acknowledgements}
The author would like to thank Larry Rolen, Walter Bridges, Andreas Mono, and Caner Nazaroglu for helpful conversations related to this project.  The author was supported by funding from the European Research Council (ERC) under the European Union's Horizon 2020 research and innovation programme (grant No. 101001179).

\section{Background and a Theorem of Vign\'{e}ras}

\subsection{Background on Hecke characters and a specialization result}
We now introduce Hecke characters related to modular forms.  For more information one can see \cite{CS}.  Let $K = \Q(\sqrt{D})$ with $D<0$ be a quadratic field of discriminant $D$ with ring of integers $\mathcal{O}_{K}$.  For an integral ideal $\m$, recall that the multiplicative congruence $\alpha \equiv^{*} 1 \pmod{\m}$ means that $v_{\mathfrak{p}}(\alpha -1) \geq v_{\mathfrak{p}}(\m)$ for all prime ideals $\mathfrak{p}|\m$, where $v_{\mathfrak{p}}(\m)$ is the valuation of $\m$ at $\mathfrak{p}$.  For an integral ideal $\m$, let $I(\mathfrak{m})$ be the group of fractional ideals of $K$ coprime to $\mathfrak{m}$.  A group homomorphism $\xi : I(\mathfrak{m}) \to \C^{\times}$ is a {\it{Hecke character}} modulo $\mathfrak{m}$ if there exists a group homomorphism $\xi_{\infty}: K^{\times}/\Q^{\times} \to \C^{\times}$ such that
\begin{equation*}
\xi(\alpha \mathcal{O}_{K}) = \xi_{\infty}(\alpha)
\end{equation*}
for all $\alpha \in K^{\times}$ such that $\alpha \equiv^{*} 1 \pmod{\mathfrak{m}}$.  When $K$ is an imaginary quadratic field we have $\xi_{\infty}(\alpha) = \left(\frac{\alpha}{|\alpha|} \right)^{m}$ for some integer $m$.  A Hecke character modulo $\m$ induces an ordinary character, $\xi_{\m}$, on $\left( \mathcal{O}_{K}/\m \right)^{\times}$ so that $\xi(\alpha \mathcal{O}_{K}) = \xi_{\m}(\alpha) \xi_{\infty}(\alpha)$ for any $\alpha \in K^{\times}$.  We now introduce the functions on integral ideals relevant to the definition of Jacobi forms with CM.  Let ${\rm Cl}(K)$ be the ideal class group of $K$.  For each class $[\a] \in {\rm Cl}(K)$, we can choose an ideal $\a_{0} \in [\a]^{-1}$, coprime to $\m$, and write $\a_0 = w_1 \Z + w_2 \Z$.  An ideal $\a$ is integral if and only if $\a \a_0 = \alpha \mathcal{O}_{K}$ for some  $\alpha \in \a_0$.  From this we see $N(\a) = \frac{N(\alpha)}{N(\a_0)} = \frac{N(n_1 w_1 + n_2 w_2)}{N(\a_0)}$ for some $n = \begin{pmatrix} n_1 \\ n_2 \end{pmatrix} \in \Z^2$, which we recognize as the binary quadratic form associated to the class of $\a_0$.  Explicitly, we have
\begin{equation*}
Q_{\a_0}(n) = an_{1}^2 + bn_{1} n_{2} + cn_{2}^2,
\end{equation*}
with the integers $a, b$, and $c$ defined by
\begin{equation*}
a= \frac{w_1 \overline{w_1}}{N(\a_0)}, \quad b = \frac{w_1 \overline{w_2} + \overline{w_1} w_2}{N(\a_0)}, \quad c= \frac{w_2 \overline{w_2}}{N(\a_0)}.
\end{equation*}  
We will denote this by $Q_{\a_0}(n)$ or $Q_{C}(n)$ when $\a_0$ is in the class $C$ and will let $B_{\a_0}(n,m) := Q_{\a_0}(n+m) - Q_{\a_0}(n) -Q_{\a_0}(m)$ be the associated bilinear form.  Following the argument above, we define the function 
\begin{equation*}
\beta(\a, m) = \beta \left( \alpha \a_{0}^{-1}, m\right) := B_{\a_{0}}(n,m),
\end{equation*}
where $\alpha = n_1 w_1 + n_2 w_2$.  We will allow some abuse of notation by writing $B_{\a_0}(\alpha, m) = B_{\a_{0}}(n, m)$ and $Q_{\a_0}(\alpha) = Q_{\a_0}(n)$. Further, for $a \in \R^2$ and $\alpha = n_1 w_1 + n_2 w_2$ as above, we define the following ways to allow shifts in our functions on ideals:
\begin{align} \label{idealfunctions}
\begin{split}
N(\a +a) &= N((\alpha) \a_{0}^{-1} +a) = \frac{N(\alpha +a)}{N(\a_0)} = Q_{\a_0}(\alpha +a) = Q_{\a_0}(n +a), \\
\beta(\a +a, m) &= B_{\a_0}(\alpha +a, m) = B_{\a_0}(n+a, m) \\
\xi(\a +a) &= \xi_{\m}(\a) \xi_{\infty}(\a +a) = \xi_{\m}(\a) \xi_{\infty}(\a_0)^{-1} \xi_{\infty}(\alpha +a),
\end{split}
\end{align}
where in each case we can also let $\alpha + a = (n_1 +a_1) w_1 + (n_2 + a_2) w_2$.  

We highlight a special case of Theorem \ref{Main} where $k=1$ and $\xi$ is primitive where the theta functions studied here satisfy an additional Fricke-like involution due to their behavior with respect to the Fourier transform.  To state this theorem we introduce some notation.  Let $S = \begin{pmatrix} 0 & -1 \\ 1 & 0 \end{pmatrix}$ and let $N$ be the level of the quadratic form $Q$. Further, let $Q^{*}(\cdot) = \frac{Q(\cdot)}{N}$ and let $B^{*} (\cdot, \cdot) = \frac{B(\cdot, \cdot)}{N}$ be its associated bilinear form.
\begin{Thm} \label{P}
Assume the notation above and let $\phi_{\xi, C}^{*}(z; \tau) \coloneqq \sum_{\alpha \in \a_0} \xi_{\m}(\alpha) e^{2 \pi i B_{C}^{*}(\alpha, z)} q^{Q_{C}^{*}(\alpha)}$.  Then we have
\begin{enumerate}
\item $\phi_{\xi,C}^{*}(z; \tau +1) = e^{\frac{2 \pi i}{N}}\phi_{\xi,C}^{*}(z; \tau)$.
\item $\phi_{\xi,C}^{*}(z + \lambda \tau + \mu; \tau) = e^{-2 \pi i B_{C}^{*}(\lambda,z)} q^{-Q_{C}^{*}(\lambda)} \phi_{\xi,C}^{*}(z; \tau), \quad \lambda, \mu \in \Z^2$.
\item $\phi_{\xi, C}^{*}\left(\frac{z}{\tau}; -\frac{1}{\tau} \right) = -i \tau e^{2 \pi i \frac{Q_{C}^{*}(z)}{\tau}}  \xi_{\m}(S) \phi_{\xi,C}^{*}(z;\tau)$.
\end{enumerate}
\end{Thm}

For $\alpha \in \a_0$, one should view $\alpha^{k-1} = (n_1 w_1 + n_2 w_2)^{k-1}$ as a spherical polynomial with respect to the quadratic form $Q_{\a_0}$.  In fact it is known that all such spherical polynomials of degree $k-1$ are linear combinations of $\alpha^{k-1}$ and $\overline{\alpha}^{k-1}$ \cite{Z2}.  We therefore will show that the constructed forms in the examples specialize to modular forms with CM.  For simplicity in the following lemma we will restrict ourselves to fields with class number one and diagonal quadratic form.  Explicitly, for $K=\Q(\sqrt{D})$ with $D<0$ square-free we have $\mathcal{O}_{K} = \Z[\rho]$ with
\begin{equation*}
\rho = \begin{cases} \sqrt{D} & D \equiv 2, 3 \pmod{4} \\ \frac{1+\sqrt{D}}{2} & D \equiv 1 \pmod{4}. \end{cases}
\end{equation*}
Therefore, we write $\alpha$ as either $n_1 + \sqrt{D}n_2$ or $\frac{n_1 + \sqrt{D}n_2}{2}$ with $n_1 \equiv n_2 \pmod{2}$ depending on the congruence class of $D$ modulo $4$.  We then have $Q(\alpha) = \alpha \overline{\alpha}$.  In order to state the next theorem we need slightly more notation.  Let $\xi_{k-1}$ denote a Hecke character modulo $\m$ with infinity type $\xi_{\infty}(\alpha) = (\alpha/|\alpha|)^{k-1}$.  Then the following Lemma is clear from definitions and is what is meant in this work when we say a Jacobi form with CM specializes to a modular form with CM.
\begin{Lemma} \label{S}
Let $K = \Q(\sqrt{D})$ have class number one, $Q(\cdot)$ its quadratic form and $B(\cdot, \cdot)$ the associated bilinear form.  Then
\begin{align*}
\phi_{\xi_{0}}(z; \tau) &= \sum_{\alpha \in \mathcal{O}_{K}} \xi_{\m} e^{2 \pi i B(\alpha, z)} q^{Q(\alpha)}.
\end{align*}
Further, if we let $z = z_1 + \rho z_2$, then
\begin{equation*}
D_{z}^{k-1} \left. \left( \phi_{\xi_{0}}(z; \tau) \right) \right|_{t=0}  = C\theta_{\xi_{k-1}}(\tau),
\end{equation*}
for some nonzero constant $C$.
\end{Lemma}
\begin{rmk}
One can view the formula $D_{z}^{k-1} \left. \left( \phi_{\xi_{0}}(z; \tau) \right) \right|_{t=0}  = C\theta_{\xi_{k-1}}(\tau)$ as a generalization of the well known formula $D_{z}(\theta(z; \tau))|_{z=0} =  \eta(\tau)^3$.
\end{rmk}
\subsection{A Theorem of Vign\'{e}ras}
In \cite{V} Vign\'{e}ras identified the differential equations and growth conditions necessary in order to use a Poisson summation argument to prove the modularity of a wide class of theta functions.  The version of her result presented here can be found in Chapter 8 of \cite{BFOR}.
\begin{thm*}[Vign\'{e}ras]
Let $Q(x) = \frac{1}{2} x^{t}Ax$ be a quadratic form of type $(r-s,s), L \subset \R^r$ a lattice on which $Q$ takes integral values, $L':= \{x \in \R^r : B(x,\ell) \in \Z \ \forall \ell \in L \}$ its dual lattice and $p: \R^r \to \C$ a Schwartz function satisfying
\begin{equation*}
\left( \mathcal{E} - \frac{\Delta_{Q}}{4 \pi} \right)p = \lambda p
\end{equation*}
for some $\lambda \in \Z$, where $\mathcal{E}$ is the Euler operator
\begin{equation*}
\mathcal{E} := \sum_{j=1}^r x_{j} \frac{\partial}{\partial x_j}
\end{equation*}
and $\Delta_Q$ is the Laplacian associated to $Q$.  Then the theta function
\begin{equation*}
\theta_{\mu}(z; \tau) := v^{-\frac{\lambda}{2}} \sum_{n \in \mu +L} p \left( \sqrt{v} \left(n + \frac{y}{v} \right) \right) e^{2 \pi i B(n,z)} q^{Q(n)}
\end{equation*}
satisfies the modular transformations
\begin{align*}
\theta_{\mu} \left( \frac{z}{\tau}; -\frac{1}{\tau} \right) &= \frac{(-i\tau)^{\lambda +\frac{r}{2}}}{\sqrt{|L'/L|}} e^{\pi i Q(A^{-1} A^{*})} \sum_{\nu \in L'/L} e^{-2 \pi i B(\mu, \nu) + 2 \pi i \frac{Q(z)}{\tau}} \theta_{\nu}(z; \tau) \\
\theta_{\mu}(z; \tau + 1) &= e^{2 \pi i Q\left(\mu + \frac{1}{2} A^{-1} A^{*} \right)} \theta_{\mu}(z; \tau),
\end{align*}
where $A^{*} := (A_{1,1}, \dots, A_{r,r})^{t}$.
\end{thm*}

\section{Proofs of Theorems \ref{Main} and \ref{P}}
\subsection{Proof of Theorem \ref{Main}}
We recall that with $z = a \tau + b$ we have
\begin{equation*}
\theta_{\xi}(z; \tau) = \sum_{\a \ \text{integral}} \xi(\a + a) N(\a + a)^{\frac{k-1}{2}} e^{2 \pi i \beta(\a, b)} q^{N(\a) + \beta(\a, a)},
\end{equation*}
where $\xi$ is a Hecke character modulo $\m$ with infinity type $\xi_{\infty}(\alpha) = (\alpha/|\alpha|)^{k-1}$.  As stated in Section 1, for each class $C \in Cl(K)$ we choose an ideal $\a_0 \in C^{-1}$ coprime to $\m$ and write $\a_{0} = w_1 \Z + w_2 \Z$.  The ideal $\a \in C$ being integral is equivalent to $\a \a_0 = \alpha \mathcal{O}_{K}$ for some $\alpha \in \a_0$.  The functions in $\theta_{\xi}$ are defined so they can be written as
\begin{align*}
\xi(\a +a) &= \xi_{\m}(\a) \xi_{\infty}(\a +a) =  \xi_{\m}(\a_0)^{-1} \xi_{m}(\alpha) \xi_{\infty}(\a_0)^{-1} \xi_{\infty}(\alpha +a) = \xi(\a_{0})^{-1} \xi_{\m}(\alpha) \xi_{\infty}(\alpha +a), \\
N(\a +a) &=\frac{N(\alpha +a)}{N(\a_0)} =Q_{\a_0}(\alpha +a), \\
\beta(\a, a) &= B_{\a_0}(\alpha, a).
\end{align*}
Now instead of summing over integral ideals we will sum over $\alpha \in \a_{0}$.  If $\epsilon$ is a root of unity in $K$, then $\alpha$ and $\epsilon \alpha$ give the same $\a$ so we need to also divide by the number of roots of unity in $K$ which we denote by $w(D)$.  We now have
\begin{align*}
\theta_{\xi}(z; \tau) &= \sum_{C \in Cl(K)} \frac{1}{w(D)} \xi(\a_0)^{-1} \sum_{\alpha \in \a_0} \xi_{\m}(\alpha) \xi_{\infty}(\alpha +a) \left( \frac{N(\alpha +a)}{N(\a_0)} \right)^{\frac{k-1}{2}} e^{2 \pi i B_{\a_0}(\alpha, b)} q^{Q_{\a_0}(\alpha) + B_{\a_0}(\alpha, a)} \\
&= \sum_{C \in Cl(K)} \frac{1}{w(D)} \xi(\a_{0})^{-1} N(\a_0)^{\frac{1-k}{2}} \sum_{\alpha \in \a_0} \xi_{\m}(\alpha) (\alpha +a)^{k-1} e^{2 \pi i B_{\a_0}(\alpha, b)}q^{Q_{\a_0}(\alpha) + B_{\a_0}(\alpha, a)}.
\end{align*}
We simplify the terms appearing as exponents as
\begin{align*}
B(\alpha, b) + Q(\alpha) \tau + B(\alpha, a) \tau = B(\alpha, a \tau + b) +Q(\alpha) \tau = B(\alpha, z) +Q(\alpha) \tau
\end{align*}
so we have
\begin{equation*}
\theta_{\xi}(z; \tau) = \sum_{C \in Cl(K)} \frac{1}{w(D)} \xi(\a_0)^{-1} N(\a_0)^{\frac{1-k}{2}} \phi_{\xi, C^{-1}}(z; \tau),
\end{equation*}
where
\begin{equation*}
\phi_{\xi, C^{-1}}(z; \tau) = \sum_{\alpha \in \a_0} \xi_{\m}(\alpha) (\alpha +a)^{k-1} e^{2 \pi i B_{\a_0}(\alpha, z)} q^{Q_{\a_0}(\alpha)}.
\end{equation*}
We choose to write $\alpha \in \a$ as $\alpha = n_1 w_1 + n_2 w_2$ for $(n_1, n_2) \in \Z^2$.  With this notation we have $Q_{\a_0}(\alpha) = Q_{\a_0}(n) = a n_{1}^2 + bn_1 n_2 + c n_{2}^2$ where
\begin{equation*}
a= \frac{w_1 \overline{w_1}}{N(\a_0)}, \quad b = \frac{w_1 \overline{w_2} + \overline{w_1} w_2}{N(\a_0)}, \quad c= \frac{w_2 \overline{w_2}}{N(\a_0)}
\end{equation*} 
are integers.  We also define the spherical polynomial
\begin{equation} \label{spherical}
P_{k-1}(n) := \alpha^{k-1} =(n_1 w_1 + n_2 w_2)^{k-1}.
\end{equation}
As noted above, all spherical polynomials (of degree $k-1$) with respect to $Q_{\a_0}$ are linear combinations of $\alpha^{k-1}$ and $\overline{\alpha}^{k-1}$.  Therefore we have $\Delta_{Q_{\a_0}}(P_{k-1}(n)) =0$.  It can also be seen that for all $c \in \R$, $P_{k-1}(cn) = c^{k-1}P_{k-1}(n)$ so $\mathcal{E}(P_{k-1}(n)) = (k-1)P_{k-1}(n)$ (see for example \cite{Stopple}).  

Then by the theorem of Vign\'{e}ras if we let $\mathds{1}$ be the trivial character the function $\phi_{\mathds{1},C}(z; \tau) = \phi_{C}(z; \tau)$ satisfies the following transformation properties:
\begin{align*}
\phi_{C}(z + \lambda \tau + \mu; \tau) &= e^{-2 \pi i B_{C}(z, \lambda)}q^{-Q_{C}(\lambda)} \phi_{C}(z; \tau), \qquad \lambda \in \Z^2, \mu \in A^{-1}\Z^2, \\
\phi_{C}(z; \tau +1) &= \phi_{C} (z; \tau), \\
\phi_{C} \left(\frac{z}{\tau}; -\frac{1}{\tau} \right) &= \frac{i (-i \tau)^{k}}{\sqrt{-\det(A)}} \sum_{m \in  A^{-1} \Z^2/\Z^2} e^{\frac{2 \pi i}{\tau} Q_{C}(z+m \tau)} \phi_{C}(z +m \tau; \tau).
\end{align*}
From these properties it is shown in \cite{Choie} that $\phi_{C}(z; \tau)$ is a weight $k$ Jacobi form on $\Gamma_{0}(|D|)$ with character $\chi_{D}(\cdot) = \left(\frac{D}{\cdot} \right)$.  The function $\phi_{\xi,C}(z; \tau)$ is just $\phi_{C}$ twisted by the character $\xi_{\m}$ and so $\phi_{\xi,C}$ is a Jacobi form of weight $k$ on $\Gamma_{0}(|D| N(\m))$ with character $\chi_{D}\xi_{\m}$.  Twisting by a character is standard (see for example \cite{CS}).  Theorem \ref{Main} then follows as $\theta_{\xi}$ is just a linear combination of $\phi_{\xi,C}$ functions.

\subsection{Proof of Theorem \ref{P}}
We will again prove the result for the trivial character and then the full result will follow either by twisting by the suitable character or by inserting it in the definition of $f_{\tau}(m)$ below.  For an integer $k \geq 1$, define
\begin{equation*}
\phi_{k,C}^{a,b}(\tau) = \sum_{m \in \Z^2} P_{k-1}(m+a) e^{2 \pi i B_{C}(m+a,b)} q^{Q_{C}(m+a)}
\end{equation*}
with $P_{k-1}$ as in equation \eqref{spherical} so we have
\begin{equation*}
\phi_{C}(a \tau + b; \tau) = e^{-2 \pi i B_{C}(a,b)} q^{-Q_{C}(a)} \phi_{k,C}^{a,b}(\tau).
\end{equation*}
Define
\begin{equation*}
f_{\tau}(m) := P_{k-1}(m) e^{2 \pi i Q_{C}(m) \tau}
\end{equation*}
and recall the Fourier transform is given by
\begin{equation*}
\mathcal{F}(f_{\tau}(m))(n) = \int_{\R^2} f_{\tau}(m) e^{- 2 \pi i B_{C}(n,m)} dm.
\end{equation*}
In \cite{V} Vign\'{e}ras showed that $\mathcal{F}(f_{-\frac{1}{\tau}}(m))(n) = -i \tau^k f_{\tau}(n)$.  Let $N$ be the level of $Q_{C}$ as in Theorem \ref{P} and let $\tau \mapsto N\tau$ so we have $\mathcal{F}(f_{-\frac{1}{N\tau}}(m))(n) = -i (N\tau)^k f_{N\tau}(n)$.  Using standard properties of the Fourier transform we thus have
\begin{align*}
\mathcal{F} \left( e^{2 \pi i B_{C}(m+a, b)} f_{-\frac{1}{N\tau}}(m+a) \right)(n) &= -i(N\tau)^k e^{2\pi i B_{C}(a,n)} f_{N\tau}(n-b) \\
&= -i (N\tau)^k q^{NQ_{C}(b)} P_{k-1}(n-b) e^{2 \pi i B_{C}(n, -bN\tau+a)} q^{NQ_{C}(n)}.
\end{align*}
The function inside the Fourier transform is
\begin{align*}
e^{2 \pi i B_{C}(m+a, b)} f_{-\frac{1}{N\tau}}(m+a) &= P_{k-1}(m+a) e^{2 \pi i \left(B_{C}(m+a,b) - \frac{Q_{C}(m+a)}{N\tau} \right)} \\
&= e^{2 \pi i \left(B_{C}(a,b) - \frac{Q_{C}(a)}{N\tau} \right)} P_{k-1}(m+a) e^{2 \pi i \left(B_{C} \left(m, \frac{bN\tau -a}{N\tau} \right) - \frac{Q_{C}(m)}{N\tau} \right)}.
\end{align*}
We then use the Poisson summation formula to obtain
\begin{equation*}
\sum_{m \in \Z^2} e^{2 \pi i B_{C}(m+a,b)} f_{-\frac{1}{N\tau}}(m+a) = \frac{-i(N\tau)^k}{|A^{-1}\Z^2/\Z^2|} \sum_{n \in A^{-1} \Z^2} e^{2 \pi i B_{C}(a,n)} f_{N\tau}(n-b).
\end{equation*}
This is equivalent to
\begin{align*}
&e^{2 \pi i \left(B_{C}(a,b) - \frac{Q_{C}(a)}{N\tau} \right)} \sum_{m \in \Z^2} P_{k-1}(m+a) e^{2\pi i \left(\frac{B_{C}\left(m,\frac{bN\tau -a}{\tau}\right)}{N} - \frac{Q_{C}(m)}{N\tau} \right)} \\
&= -i N^{k-1}\tau^k e^{2 \pi i NQ_{C}(b)\tau} \sum_{n \in \Z^2} P_{k-1}\left(-\frac{n}{N}-b\right)e^{2 \pi i \left(B_{C}\left(\frac{n}{N}, bN\tau -a \right) + NQ_{C}\left(\frac{n}{N}\right) \tau \right)} \\
&= (-1)^k i \tau^k e^{2 \pi i NQ_{C}(b)\tau} \sum_{n \in \Z^2} P_{k-1}(n+Nb) e^{2 \pi i \left( \frac{B_{C}(n,bN\tau -a)}{N} + \frac{Q_{C}(n)\tau}{N} \right)}.
\end{align*}
We rearrange and combine the pre-factors as
\begin{equation*}
e^{2\pi i \left(NQ_{C}(b)\tau + \frac{Q_{C}(a)}{N\tau} - B_{C}(a,b) \right)} = e^{2 \pi i \frac{Q_{C}(bN\tau -a)}{N\tau}}
\end{equation*}
and set $z = bN\tau -a$ which allows us to obtain
\begin{equation} \label{Poisson1}
\sum_{m \in \Z^2} P_{k-1}(m+a) e^{2 \pi i \left( B_{C}^{*} \left(m, \frac{z}{\tau} \right) - \frac{Q_{C}^{*}(m)}{\tau}\right)} = (-1)^k i \tau^k e^{2 \pi i \frac{Q_{C}^{*}(z)}{\tau}} \sum_{n \in \Z^2} P_{k-1}(n+Nb) e^{2 \pi i (B_{C}^{*}(n,z) + Q_{C}^{*}(n)\tau)}.
\end{equation}
In particular, when $k=1$ we have $\phi_{C}^{*} \left(\frac{z}{\tau}; -\frac{1}{\tau} \right) = -i\tau e^{2 \pi i \frac{Q_{C}^{*}(z)}{\tau}} \phi_{C}^{*}(z;\tau)$.  Twisting by the character $\xi_{\m}$ as in \cite{CS} completes the proof of part (3) of Theorem \ref{P}.  The first two parts follow immediately from the theorem of Vign\'{e}ras or simply by making a suitable change of variable in the sum defining $\phi_{\xi, C}^{*}$.

\section{Examples and applications}
\subsection{Even powers of $\eta(\tau)$ with complex multiplication}
In \cite{Serre} Serre explicitly constructed Hecke characters $\xi$ so that $\chi_{D} \xi_{\m} = \epsilon^k$ for positive integers $k$ where $\epsilon$ is the character of the Dedekind eta function $\eta(\tau)$.  In doing so Serre showed that the only even powers $\eta^{r}(\tau)$ with complex multiplication are $r \in \{2, 4, 6, 8, 10, 14, 26\}$.  In this section we will construct Jacobi forms with complex multiplication with nearly the same Hecke characters constructed by Serre which sit above the powers of $\eta$ with CM.  There are some small differences between the characters constructed here and the ones Serre constructed in \cite{Serre}.  First off, all of the characters here will have infinity type with exponent $0$ and then differentiating the Jacobi variable the same number of times as the exponent in Serre's work will allow us to obtain his identities.  Secondly, Serre's characters are not unitary, so they have infinity type $\xi_{\infty}(\alpha)=\alpha^{k-1}$ rather than $\xi_{\infty}(\alpha) = \left(\frac{\alpha}{|\alpha|} \right)^{k-1}$ which is just a matter of convention. Because. we are always choosing the infinity part to have exponent zero we are really just computing the character $\xi_{\m}$ on $\left(\mathcal{O}_{K}/\m \right)^{\times}$ such that $\xi_{\m}(\alpha) \alpha^{k-1}$ is Serre's character.  The examples for $r \neq 26$ allow us to reprove the celebrated Macdonald identities connected to Lie algebras \cite{M} which have also recently been reconstructed using the theory of Theta blocks by Gritsenko, Skoruppa, and Zagier \cite{GSZ}.  The only identity that appears to have been written down in a similar form is the $r=8$ identity which first appeared in \cite{S2} and then was called a theta quark in \cite{GSZ}.  

In order to state the identities we introduce the {\bf Jacobi theta function}
\begin{equation*}
\theta(z; \tau) \coloneqq q^{\frac{1}{8}} (\zeta^{\frac{1}{2}} - \zeta^{-\frac{1}{2}}) \prod_{n \geq 1} (1-q^n) (1-\zeta q^n)(1-\zeta^{-1}q^n) = \sum_{n \in \Z} \left(\frac{-4}{n} \right) \zeta^{\frac{n}{2}} q^{\frac{n^2}{8}}.
\end{equation*}
The Jacobi theta function is a Jacobi form of weight $\frac{1}{2}$, index $\frac{1}{2}$, and character $\epsilon^3$.  For completeness we also present a special Jacobi form $\theta^{*}(z; \tau) \coloneqq \frac{\eta(\tau) \theta(2z; \tau)}{\theta(z; \tau)}$ and standard shifts of both $\theta$ and $\theta^{*}$.  Note that $\theta^{*}$ is a weight $\frac{1}{2}$, index $\frac{3}{2}$ Jacobi form with character $\epsilon$.  We will also let $e(a) = e^{2 \pi i a}$ and $\zeta_{i} = e(z_i)$ throughout.  The following two propositions record Jacobi theta functions which can be viewed as building blocks for our purposes.
\begin{Prop} \label{BB}
Assume the notation above.  Then we have the following representations of the Jacobi theta functions.
\begin{align*}
\theta_{2}(z; \tau) &\coloneqq -i \theta \left(z + \frac{1}{2}; \tau \right) = \sum_{n \in \Z} \left(\frac{4}{n} \right) \zeta^{\frac{n}{2}} q^{\frac{n^2}{8}} =\frac{\eta(\tau)^2 \theta(2z; 2 \tau)}{\eta(2 \tau) \theta(z; \tau)}, \\
\theta_{3}(z; \tau) &\coloneqq -i \zeta^{\frac{1}{2}} q^{\frac{1}{8}} \theta \left( z + \frac{1}{2} + \frac{\tau}{2}; \tau \right) = \sum_{n \in \Z} \left(\frac{4}{n} \right) \zeta^{\frac{n+1}{2}} q^{\frac{(n+1)^2}{8}} = \frac{\eta \left(\frac{\tau}{2} \right) \eta(2 \tau) \theta(z; \tau) \theta(2z; \tau)}{\eta(\tau) \theta \left(z; \frac{\tau}{2} \right) \theta(2z; 2 \tau)}, \\
\theta_{4}(z; \tau) &\coloneqq - \zeta^{\frac{1}{2}} q^{\frac{1}{8}} \theta \left(z + \frac{\tau}{2}; \tau \right) = -\sum_{n \in \Z} \left( \frac{-4}{n} \right) \zeta^{\frac{n+1}{2}} q^{\frac{(n+1)^2}{8}} = \frac{\eta(\tau)^2 \theta \left(z; \frac{\tau}{2} \right)}{\eta \left(\frac{\tau}{2} \right) \theta(z; \tau)}.
\end{align*}
\end{Prop}
We have the following analogous proposition for $\theta^*$.
\begin{Prop} \label{BB2}
Assume the notation above.  Then we have the following representations of the Jacobi theta functions.
\begin{align*}
\theta^{*}(z; \tau) &= \sum_{n \in \Z} \left(\frac{12}{n} \right) \zeta^{\frac{n}{2}} q^{\frac{n^2}{24}} = \frac{\eta(\tau) \theta(2z; \tau)}{\theta(z; \tau)}, \\
\theta_{2}^{*}(z; \tau) &\coloneqq -i\theta^{*} \left(z + \frac{1}{2}; \tau \right) = \sum_{n \in \Z} \left(\frac{-12}{n} \right) \zeta^{\frac{n}{2}} q^{\frac{n^2}{24}} = \frac{\eta(2 \tau) \theta(z; \tau) \theta(2z; \tau)}{\eta(\tau) \theta(2z; 2 \tau)}, \\
\theta_{3}^{*}(z; \tau) &\coloneqq i \zeta^{\frac{3}{2}} q^{\frac{3}{8}} \theta^{*} \left(z + \frac{1}{2} + \frac{\tau}{2}; \tau \right) = - \sum_{n \in \Z} \left(\frac{-12}{n} \right) \zeta^{\frac{n+3}{2}} q^{\frac{(n+3)^2}{24}} = \frac{\eta(\tau)^2  \theta \left( z ; \frac{\tau}{2} \right) \theta(2z; 2 \tau)}{\eta \left(\frac{\tau}{2} \right) \eta(2 \tau) \theta(z; \tau)}, \\
\theta_{4}^{*}(z; \tau) &\coloneqq \zeta^{\frac{3}{2}} q^{\frac{3}{8}} \theta^{*} \left( z + \frac{\tau}{2}; \tau \right) = \sum_{n \in \Z} \left(\frac{12}{n} \right) \zeta^{\frac{n+3}{2}} q^{\frac{(n+3)^2}{24}} = \frac{\eta \left(\frac{\tau}{2} \right) \theta(z; \tau) \theta(2z; \tau)}{\eta(\tau) \theta \left(z; \frac{\tau}{2} \right)}.
\end{align*}
\end{Prop}
For convenience we supress the dependence on $\tau$ whenever that variable is not weighted.  For example, $\theta(z; \tau) \theta(2z; \tau)$ will be written as $\theta(z) \theta(2z)$.  We will also denote $\phi_{\xi}^{*}$ by $\phi_{r}$ in this section where the Jacobi form specializes to $\eta(\tau)^r$.  In all cases except $r=26$ we omit proofs as they can be found in many places \cite{M,GSZ}.  In all cases the corollaries are obtained by applying $D_{z}^{\frac{r}{2}-1} \left[\phi_{r}(z;\tau)\right]_{z=0}$ where either $z=z_1 +iz_2$ or $z=\frac{z_1 +\sqrt{-3}z_2}{2}$ depending on if the form has CM by $\Q(i)$ or $\Q(\sqrt{-3})$ respectively.

\subsubsection{$r=2$}
Let $K=\Q(i)$ and let $\mathcal{O}_{K}$ be its ring of integers.  The Hecke character $\xi$ can be chosen to be either of order $4$ and conductor $6$ described in \cite{Serre}.  When $\alpha = n_1 + in_2$ it is known that $\xi_{6}(\alpha) = \left(\frac{12}{n_1 n_2}\right)$ (see page 59 of \cite{CS}).  We have quadratic form $Q^{*}(\alpha) = \frac{n_{1}^2 + n_{2}^2}{24}$ and so we find
\begin{equation*}
\phi_{2}(z; \tau) = \sum_{\alpha \in \mathcal{O}_{K}} \xi_{6}(\alpha) e^{2 \pi i B^{*}(\alpha, z)} q^{Q^{*}(\alpha)} = \sum_{n_1, n_2 \in \Z} \left(\frac{12}{n_1 n_2} \right) e^{2 \pi i \frac{n_1 z_1 + n_2 z_2}{12}} q^{\frac{n_{1}^2 + n_{2}^{2}}{24}}.
\end{equation*}
\begin{Prop} \label{2}
Assume the notation above.  Then
\begin{align*}
\phi_{2}(z; \tau) &= \sum_{n_1, n_2 \in \Z} \left(\frac{12}{n_1 n_2} \right) \zeta_{1}^{\frac{n_1}{12}} \zeta_{2}^{\frac{n_2}{12}} q^{\frac{n_{1}^2 + n_{2}^{2}}{24}} \\
&= \theta^{*} \left(\frac{z_1}{6} \right) \theta^{*} \left( \frac{z_2}{6} \right)
\end{align*}
\end{Prop}
We define the numbers $\tau_{2}(n)$ by
\begin{equation*}
\sum_{n \equiv 1 \pmod{12}} \tau_{2}(n) q^{\frac{n}{12}} = q^{\frac{1}{12}}\prod_{n \geq 1} (1-q^n)^2.
\end{equation*}
Then by setting $z_1 = z_2 =0$ above we have the following.
\begin{Cor}
Assume the notation above.  Then
\begin{equation*}
\tau_{2}(n) =\frac{1}{4} \sum_{\substack{ n_1, n_2 \in \Z \\ n_{1}^2 + n_{2}^2 =2n}} \left( \frac{12}{n_1 n_2} \right).
\end{equation*}
\end{Cor}

\subsubsection{$r=4$}

Let $K = \Q(\sqrt{-3})$ and $\mathcal{O}_{K}$ its ring of integers.  In \cite{Serre} the Hecke character with conductor $\m = 2 \sqrt{-3}$ is considered such that for all $\a$ coprime to $\m$, $\xi(\a) = \alpha$ where $\alpha$ generates $\a$ such that $ \alpha \equiv^{*} 1 \pmod{2 \sqrt{-3}}$.  The integral ideal $(\alpha)$ is coprime to $(2 \sqrt{-3})$ if $\alpha = \frac{n_1 + n_2 \sqrt{-3}}{2}$ with $n_1 \not\equiv 0 \pmod{3}$ and either both $n_1$ and $n_2$ odd or both $n_1$ and $n_2$ are even and $n_1 \not\equiv n_2 \pmod{4}$. The $\alpha \equiv^* 1 \pmod{2 \sqrt{-3}}$ are of the form $\alpha = 6t+1 + 2r\sqrt{-3}$ or $6t+4 +(2r+1)\sqrt{-3}$ with $t, r \in \Z$.  For each $\alpha$ coprime to $\m$ there exists a unique $0 \leq m \leq 5$ such that $\rho^m \alpha \equiv^* 1 \pmod{2 \sqrt{-3}}$ with $\rho = \frac{1+\sqrt{-3}}{2}$.  We define $\xi_{2\sqrt{-3}}(\alpha) = \rho^m$.  Explicitly, when $n_1 \equiv n_2 \equiv 0 \pmod{2}$ we have $\xi_{\m}(\alpha) = \pm 1$ where $\pm \alpha \equiv^* 1 \pmod{2 \sqrt{-3}}$ and when $n_1 \equiv n_2 \equiv 1 \pmod{2}$ we have
\begin{equation*}
\xi_{\m}(\alpha) = \begin{cases} e\left( \frac{n_1 -3n_2}{12} \right) & \text{if} \ \left(\frac{12}{n_1} \right) \left(\frac{-4}{n_2}\right) =1 \\
e\left( \frac{-n_1 -3n_2}{12} \right) & \text{if} \ \left(\frac{12}{n_1} \right) \left(\frac{-4}{n_2}\right) =-1.
\end{cases}
\end{equation*}
When $n_1$ and $n_2$ are both even and $n_1 \not\equiv n_2 \pmod{4}$ we can equivalently write $\xi_{\m}(\alpha)= \left(\frac{2n_{1}}{3} \right)$.  We thus have
\begin{align*}
\phi_{4}(z; \tau) &= \sum_{\alpha \in \mathcal{O}_{K}} \xi_{2 \sqrt{-3}}(\alpha) e^{2 \pi i B^{*}(\alpha, z)} q^{Q^{*}(\alpha)} \\
&= \sum_{\substack{ n_1, n_2 \equiv 0 \pmod{2} \\ n_1 \not\equiv n_2 \pmod{4}}} \left( \frac{2 n_1}{3} \right) e^{2 \pi i \frac{n_1 z_1 + 3n_2 z_2}{12}} q^{\frac{n_{1}^2 + 3 n_{2}^2}{24}}  \\
&+ \sum_{\left(\frac{12}{n_1} \right) \left(\frac{-4}{n_2} \right) =1}\left(\frac{12}{n_1} \right) \left(\frac{-4}{n_2} \right) e\left(\frac{n_1 -3n_2}{12}\right) e^{2 \pi i \frac{n_1 z_1 + 3n_2 z_2}{12}} q^{\frac{n_{1}^2 + 3n_{2}^2}{24}} \\
&+\sum_{\left(\frac{12}{n_1} \right) \left(\frac{-4}{n_2} \right) =-1}\left(\frac{12}{n_1} \right) \left(\frac{-4}{n_2} \right) e\left(\frac{n_1 +3n_2}{12}\right) e^{2 \pi i \frac{n_1 z_1 + 3n_2 z_2}{12}} q^{\frac{n_{1}^2 + 3n_{2}^2}{24}}.
\end{align*}
\begin{Prop}
Assume the notation above.  Then
\begin{align*}
&\sum_{\substack{ n_1, n_2 \equiv 0 \pmod{2} \\ n_1 \not\equiv n_2 \pmod{4}}} \left( \frac{2 n_1}{3} \right) \zeta_{1}^{\frac{n_1}{12}} \zeta_{2}^{\frac{n_2}{4}} q^{\frac{n_{1}^2 + 3 n_{2}^2}{24}} \\
&= \sum_{n_1, n_2 \in \Z} \left( \frac{n_1}{3} \right) \left(\frac{4}{n_{1}^2 + n_{2}^2} \right) \zeta_{1}^{\frac{n_1}{6}} \zeta_{2}^{\frac{n_2}{2}} q^{\frac{n_{1}^2 + 3 n_{2}^2}{6}} \\
&= \frac{1}{2} \left[\theta_{3}^{*} \left(\frac{z_1}{6} \right) \theta_{3} \left( \frac{z_2}{2} \right) + \theta_{4}^{*} \left( \frac{z_1}{6} \right) \theta_{4} \left( \frac{z_2}{2} \right) \right],
\end{align*}
and
\begin{align*}
&\sum_{\left(\frac{12}{n_1} \right) \left(\frac{-4}{n_2} \right) =1}\left(\frac{12}{n_1} \right) \left(\frac{-4}{n_2} \right) e\left(\frac{n_1 -3n_2}{12}\right) \zeta_{1}^{\frac{n_1}{12}} \zeta_{2}^{\frac{n_2}{4}}q^{\frac{n_{1}^2 + 3n_{2}^2}{24}} \\
&+\sum_{\left(\frac{12}{n_1} \right) \left(\frac{-4}{n_2} \right) =-1}\left(\frac{12}{n_1} \right) \left(\frac{-4}{n_2} \right) e\left(\frac{-n_1 +3n_2}{12}\right) \zeta_{1}^{\frac{n_1}{12}} \zeta_{2}^{\frac{n_2}{4}} q^{\frac{n_{1}^2 + 3n_{2}^2}{24}} \\
&= \frac{1}{2} \sum_{n_1, n_2 \in \Z} \left[ \left( \frac{-12}{n_1} \right) \left( \frac{4}{n_2} \right) - \sqrt{-3} \left( \frac{12}{n_1} \right) \left( \frac{-4}{n_2} \right) \right] \zeta_{1}^{\frac{n_1}{12}} \zeta_{2}^{\frac{n_2}{4}} q^{\frac{n_{1}^2  +3 n_{2}^2}{24}} \\
&= \frac{1}{2} \left[ \theta_{2}^{*} \left(\frac{z_1}{6} \right) \theta_{2} \left( \frac{z_2}{2} \right) - \sqrt{-3} \theta^{*} \left( \frac{z_1}{6} \right) \theta \left( \frac{z_2}{2} \right) \right].
\end{align*}
In particular, we have
\begin{align*}
\phi_{4}(z; \tau) &= \frac{1}{2} \left[ \theta_{3}^{*} \left(\frac{z_1}{6} \right) \theta_{3} \left( \frac{z_2}{2} \right) + \theta_{4}^{*} \left( \frac{z_1}{6} \right) \theta_{4} \left( \frac{z_2}{2} \right) \right] \\
&+ \frac{1}{2} \left[ \theta_{2}^{*} \left(\frac{z_1}{6} \right) \theta_{2} \left(\frac{z_2}{2} \right) - \sqrt{-3} \theta^{*} \left(\frac{z_1}{6} \right) \theta \left(\frac{z_2}{2} \right) \right].
\end{align*}
\end{Prop}
\begin{rmk}
If one wants a simpler Jacobi form of this shape that lives above $\eta(\tau)^4$, then one can of course take 
\begin{equation*}
\theta^{*} \left( \frac{z_1}{6} \right) \theta \left( \frac{z_2}{2} \right) = \sum_{n_1, n_2 \in \Z} \left(\frac{12}{n_1} \right) \left(\frac{-4}{n_2} \right) \zeta_{1}^{\frac{n_1}{12}} \zeta_{2}^{\frac{n_2}{4}} q^{\frac{n_{1}^2 + 3 n_{2}^2}{24}}.
\end{equation*}
\end{rmk}
Define the coefficients $\tau_{4}(n)$ by
\begin{equation*}
\sum_{n \equiv 1 \pmod{6}} \tau_{4}(n) q^{\frac{n}{6}} = q^{\frac{1}{6}} \prod_{n \geq 1}(1-q^n)^4.
\end{equation*}
Then we have the following.
\begin{Cor}
Assume the notation above.  Then
\begin{align*}
\tau_{4}(n) &=\sum_{\substack{n_1 \geq 0 \\ n_{1}^2 + 3 n_{2}^2 =n}} \left(\frac{n_1}{3} \right) \left(\frac{4}{n_{1}^2 + n_{2}^2} \right) n_1 \\
&= -\frac{1}{6} \sum_{\substack{ n_1 \geq 0  \\ n_{1}^2 +3 n_{2}^2 =4n}}  \left( \frac{-12}{n_1} \right) \left( \frac{4}{n_2} \right)n_1 - 3 \left( \frac{12}{n_1} \right) \left( \frac{-4}{n_2} \right)n_2.
\end{align*}
\end{Cor}

\subsubsection{$r=6$}
Let $K = \Q(i)$ and $\mathcal{O}_{K}$ its ring of integers.  We consider a Hecke character with conductor $\m = 2$.  If $\alpha = n_1 + i n_2$, then $\alpha$ is coprime to $\m$ as long as $n_1 \not\equiv n_2 \pmod{2}$ and $\alpha \equiv^* 1 \pmod{2}$ when $n_1$ is odd and $n_2$ is even.  We therefore have
\begin{equation*}
\xi_{\m}(\alpha) = \begin{cases} 0 & \text{if} \ n_1 \equiv n_2 \pmod{2} \\ (-1)^{n_{2}} & \text{if} \ n_1 \not\equiv n_2 \pmod{2}.
\end{cases}
\end{equation*}.
Thus
\begin{equation*}
\phi_{6}(z; \tau) = \sum_{\alpha \in \mathcal{O}_{K}} \xi_{2}(\alpha) e^{2 \pi i B^{*}(\alpha, z)} q^{Q^{*}(\alpha)} = \sum_{n_1, n_2 \in \Z} (-1)^{n_2} \left( \frac{4}{n_{1}^2 + n_{2}^2} \right) e^{2 \pi i \frac{n_1 z_1 + n_2 z_2}{2}} q^{\frac{n_{1}^2 + n_{2}^2}{4}}.
\end{equation*}
\begin{Prop}
Assume the notation above.  Then
\begin{align*}
\phi_{6}(z;\tau) &= \sum_{n_1, n_2 \in \Z} (-1)^{n_{2}} \left(\frac{4}{n_{1}^2 + n_{2}^2} \right) \zeta_{1}^{\frac{n_{1}}{2}} \zeta_{2}^{\frac{n_2}{2}} q^{\frac{n_{1}^2 + n_{2}^2}{4}} \\
&= \theta \left(\frac{z_1 + z_2 }{2} \right) \theta \left( \frac{z_1 - z_2}{2} \right).
\end{align*}
\end{Prop}
Define the coefficients $\tau_{6}(n)$ by
\begin{equation*}
\sum_{n \equiv 1 \pmod{4}} \tau_{6}(n) q^{\frac{n}{4}} = q^{\frac{1}{4}} \prod_{n \geq 1} (1-q^n)^6.
\end{equation*}
We then have the following.
\begin{Cor}
Assume the notation above.  Then
\begin{align*}
\tau_{6}(n) &= \frac{1}{4} \sum_{\substack{n_1, n_2 \in \Z \\ n_{1}^2 + n_{2}^2 =n}} (-1)^{n_2} \left( \frac{4}{n_{1}^2 + n_{2}^2} \right) (n_{1} -in_2)^2 \\
&= \sum_{\substack{n_1 \geq 0 \\ n_2 >0 \\ n_{1}^2 + n_{2}^2 =n}} (-1)^{n_2} \left( \frac{4}{n_{1}^2 + n_{2}^2}\right) (n_1 +n_2)(n_1 -n_2).
\end{align*}
\end{Cor}

\subsubsection{$r=8$}
Let $K=\Q(\sqrt{-3})$, let $\mathcal{O}_{K}$ be its ring of integers, and cosider a Hecke character with conductor $\m = \sqrt{-3}$.  Therefore, $\alpha = \frac{n_1 + \sqrt{-3}n_2}{2}$ with $n_1 \equiv n_2 \pmod{2}$ is coprime to $\m$ when $n_1 \not\equiv 0 \pmod{3}$ and $\alpha \equiv^* 1 \pmod{\sqrt{-3}}$ when $n_1 \equiv 2 \pmod{3}$.  We let
\begin{equation*}
\xi_{\m}(\alpha) = \left(\frac{2n_1}{3} \right).
\end{equation*}
Thus
\begin{equation*}
\phi_{8}(z; \tau) = \sum_{\alpha \in \mathcal{O}_{K}} \xi_{\sqrt{-3}}(\alpha) e^{2 \pi i B^{*}(\alpha, z)} q^{Q^{*}(\alpha)} = \sum_{\substack{ n_{1}, n_2 \in \Z \\ n_2 \equiv n_2 \pmod{2}}} \left(\frac{2n_1}{3} \right) e^{2 \pi i \frac{n_1 z_1 + 3 n_2 z_2}{6}} q^{\frac{n_{1}^2 + 3 n_{2}^2}{12}}.
\end{equation*}

\begin{Prop} \label{8}
Assume the notation above.  Then
\begin{align*}
\phi_{8}(z; \tau) &=\sum_{\substack{ n_{1}, n_2 \in \Z \\ n_1 \equiv n_2 \pmod{2}}} \left(\frac{2n_1}{3} \right) \zeta_{1}^{\frac{n_{1}}{6}} \zeta_{2}^{\frac{n_2}{2}} q^{\frac{n_{1}^2 + 3 n_{2}^2}{12}} \\
&=  \frac{\theta\left( \frac{z_1 + 3z_2}{6} \right)\theta\left( \frac{z_1 -3z_2}{6} \right) \theta \left(\frac{z_1}{3} \right)}{\eta(\tau)}.
\end{align*}
\end{Prop}
\begin{rmk}
This identity is equivalent to the MacDonald identity for the root system $A_2$.  We can identify the positive roots $r_1 = \frac{z_1 + 3z_2}{6}$, $r_2= \frac{z_1 - 3z_2}{6}$, and $r_1 + r_2 = \frac{z_1}{3}$ with the arguments in the theta functions.
\end{rmk}
Define the coefficients $\tau_{8}(n)$ by
\begin{equation*}
\sum_{n \equiv 1 \pmod{3}} \tau_{8}(n) q^{\frac{n}{3}} = q^{\frac{1}{3}}\prod_{n \geq 1}(1-q^{n})^8.
\end{equation*}
We then have the following corollary.
\begin{Cor}
Assume the notation above.  Then
\begin{align*}
\tau_{8}(n) &= \frac{1}{48}\sum_{\substack{n_1, n_2 \in \Z \\ n_1 \equiv n_2 \pmod{2} \\ n_{1}^2 + 3n_{2}^2 =4n}} \left(\frac{2n_1}{3} \right) (n_1 - \sqrt{-3}n_2)^3 \\
&= \frac{1}{48} \sum_{\substack{n_1 \geq 0 \\ n_1 \equiv n_2 \pmod{2} \\ n_{1}^2 + 3n_{2}^2 =4n}} \left(\frac{2n_1}{3} \right) 2n_1 (n_1 +3n_2)(n_1 -3n_2).
\end{align*}
\end{Cor}

\subsubsection{$r=10$}
Let $K=\Q(i)$ and let $\mathcal{O}_{K}$ be its ring of integers.  We consider the difference of two characters each with conductor $\m=6$.  Therefore $\alpha = n_1 + i n_2$ is coprime to $\m$ when $n_1 \not\equiv n_2 \pmod{2}$ and $n_1$ and $n_2$ are not both $0 \pmod{3}$ and $\alpha \equiv^* 1 \pmod{6}$ when $(n_1, n_2) \equiv (1,0) \pmod{6}$.  We consider the characters $\xi_{6}^{\pm}(\alpha) = (-1)^{n_2} (\pm i)^b$ where $\alpha \equiv (1-i)^b \pmod{3}$.  Explicitly, when $\alpha$ is coprime to $\m$ we have
\begin{equation*}
b = \begin{cases} 0 & n_1 \not\equiv n_2 \equiv 0 \pmod{3} \\
1 & n_1 \equiv -n_2 \not\equiv 0 \pmod{3} \\
2 & n_2 \not\equiv n_1 \equiv 0 \pmod{3} \\
3 & n_1 \equiv n_2 \not\equiv 0 \pmod{3},
\end{cases}
\end{equation*}
which gives
\begin{equation*}
\xi_{6}^{+}(\alpha) = \begin{cases}
(-1)^{n_2} & \text{if} \ (n_1, n_2) \equiv (\pm1, 0) \pmod{3} \\
-(-1)^{n_2} & \text{if} \ (n_1, n_2) \equiv (0, \pm 1)  \pmod{3} \\
 i (-1)^{n_2} & \text{if} \ (n_1, n_2) \equiv (\pm 1, \mp 1) \pmod{3} \\
- i (-1)^{n_2} & \text{if} \ (n_1, n_2) \equiv (\pm 1, \pm 1) \pmod{3}
\end{cases}
\end{equation*}
and 
\begin{equation*}
\xi_{6}^{-}(\alpha) = \begin{cases}
(-1)^{n_2} & \text{if} \ (n_1, n_2) \equiv (\pm1, 0) \pmod{3} \\
-(-1)^{n_2} & \text{if} \ (n_1, n_2) \equiv (0, \pm 1)  \pmod{3} \\
-i (-1)^{n_2} & \text{if} \ (n_1, n_2) \equiv (\pm 1, \mp 1) \pmod{3} \\
i (-1)^{n_2} & \text{if} \ (n_1, n_2) \equiv (\pm 1, \pm 1) \pmod{3}
\end{cases}
\end{equation*}
for all $\alpha$ coprime to $\m$.  We define
\begin{equation*}
\phi_{10}(z; \tau) = \frac{i}{2} \left( \phi_{\xi^{+}}^{*}(z; \tau) - \phi_{\xi^{-}}^{*}(z; \tau) \right).
\end{equation*}
\begin{Prop}
Assume the notation above.  Then
\begin{align*}
\phi_{10}(z; \tau) &= \sum_{(n_1, n_2) \equiv (\pm 1, \pm1) \pmod{3}} (-1)^{n_2} \left( \frac{4}{n_1 +n_2} \right) \zeta_{1}^{\frac{n_1}{6}} \zeta_{2}^{\frac{n_2}{6}} q^{\frac{n_{1}^2 + n_{2}^2 }{12}} \\
&-\sum_{(n_1, n_2) \equiv (\pm 1, \mp1) \pmod{3}} (-1)^{n_2} \left( \frac{4}{n_1 +n_2} \right) \zeta_{1}^{\frac{n_1}{6}} \zeta_{2}^{\frac{n_2}{6}} q^{\frac{n_{1}^2 + n_{2}^2 }{12}} \\
&= \sum_{n_1, n_2 \in \Z} (-1)^{n_2} \left( \frac{4}{n_{1} + n_{2}} \right)  \left[ \substack{\zeta_{1}^{\frac{3n_1 +1}{6}} \zeta_{2}^{\frac{3n_2 +1}{6}} +\zeta_{1}^{\frac{-3n_1 -1}{6}} \zeta_{2}^{\frac{-3n_2 -1}{6}} \\ -\zeta_{1}^{\frac{3n_1 +1}{6}} \zeta_{2}^{\frac{-3n_2 -1}{6}} -\zeta_{1}^{\frac{-3n_1 -1}{6}} \zeta_{2}^{\frac{3n_2 +1}{6}}} \right] q^{\frac{3 \left[ \left(n_1 + \frac{1}{3} \right)^2 + \left(n_2 + \frac{1}{3} \right)^2 \right]}{4}} \\
&=\frac{\theta \left(\frac{z_1 + z_2}{6} \right) \theta \left(\frac{z_1 - z_2}{6} \right) \theta \left(\frac{z_1}{3} \right) \theta\left( \frac{z_2}{3} \right)}{\eta(\tau)^2}.
\end{align*}
\end{Prop}
\begin{rmk}
This identity is equivalent to the MacDonald identity for the root system $B_2$.  If we let $r_1 = \frac{z_1 + z_2}{6}$ and $r_2 = \frac{z_1 - z_2}{6}$, then the arguments of the theta functions correspond to the positive roots $r_1, r_2, r_1 + r_2, r_1 - r_2$.
\end{rmk}
Define the coefficients $\tau_{10}(n)$ by
\begin{equation*}
\sum_{n \equiv 5 \pmod{12}} \tau_{10}(n) q^{\frac{n}{12}} = q^{\frac{5}{12}} \prod_{n \geq 1}(1-q^n)^{10}.
\end{equation*}
\begin{Cor}
Assume the notation above.  Then
\begin{align*}
\tau_{10}(n) &= -\frac{i}{192} \sum_{\substack{n_1, n_2 \in \Z \\ \left(n_1 +\frac{1}{3} \right)^2 + \left( n_2 + \frac{1}{3} \right)^2 = \frac{4n}{3}}} (-1)^{n_2} \left( \frac{4}{n_{1}^2 + n_{2}^2} \right) \left[ \substack{\left( 3n_1 +1 -i(3n_2 +1) \right)^4 + \left( -3n_1 -1 +i(3n_2 +1) \right)^4 \\-\left(3n_1 +i(3n_2 +1) \right)^4 - \left( -3n_1 -1 -i(3n_2 +1) \right)^4} \right] \\
&=\frac{1}{12} \sum_{\substack{n_1, n_2 \in \Z \\ \left(n_1 +\frac{1}{3} \right)^2 + \left( n_2 + \frac{1}{3} \right)^2 = \frac{4n}{3}}} (-1)^{n_{2}} \left(\frac{4}{n_{1}^2 + n_{2}^2} \right) (3n_2 -3n_1)(3n_1 +1)(3n_2 +1)(3n_1 + 3n_2 +2).
\end{align*}
\end{Cor}

\subsubsection{$r=14$}
As in the $r=10$ case, here we will use a difference of Hecke characters.  Let $K= \Q(\sqrt{-3})$ and let $\mathcal{O}_{K}$ be its ring of integers.  We again consider two different characters each with conductor $\m = 4 \sqrt{-3}$.  Thus $\alpha = \frac{n_1 + n_2 \sqrt{-3}}{2}$ is coprime to $\m$ if $n_1 \not\equiv 0 \pmod{3}$ and either $n_1 \equiv n_2 \equiv 1 \pmod{2}$ or $n_1 \equiv n_2 \equiv 0 \pmod{2}$ and $n_1 \not\equiv n_2 \pmod{4}$ and $\alpha \equiv^{*} 1 \pmod{\m}$ if $\alpha$ is of the form $\alpha = 12t+1 + 4r \sqrt{-3}$ or $\alpha = 12t + 7 + (4r+2) \sqrt{-3}$.  For each integral ideal $\a$ there exists a unique $\alpha = a + b \sqrt{-3}$ such that $a \equiv 1 \pmod{3}$, $a + b \equiv 1 \pmod{2}$ and $(\alpha) = \a$.  For each $\alpha$ coprime to $\m$ there exists a unique $0 \leq m \leq 5$ such that $ \rho^m \alpha = a + b \sqrt{-3}$ is this generator.  We define the two characters $\xi_{\m}^{\pm}$ by $\xi_{\m}^{\pm}(\alpha) = \rho^{6m} (-1)^{\frac{a \mp b -1}{2}} =(-1)^{\frac{a \mp b -1}{2}}$.  Explicitly, when $n_1 \equiv n_2 \equiv 0 \pmod{2}$ and $\alpha$ is coprime to $\m$ we have $\xi_{\m}^{\pm}(\alpha) = \left(\frac{2n_1}{3} \right) (-1)^{\frac{n_1 \mp n_2 -2}{4}}$.  When $n_1 \equiv n_2 \equiv 1 \pmod{2}$ and $\alpha$ is coprime to $\m$ we have
\begin{equation*}
\xi_{\m}^{+}(\alpha) = \begin{cases} \left(\frac{12}{n_1}\right) \left(\frac{-4}{n_2} \right)(-1)^{\frac{n_1 + n_2 -2}{4}} & n_1 \equiv n_2 \pmod{4} \\
 \left(\frac{12}{n_1}\right) \left(\frac{-4}{n_2} \right)(-1)^{\frac{n_2 -1}{2}} & n_1 \equiv -n_2 \pmod{4}
\end{cases}
\end{equation*}
and
\begin{equation*}
\xi_{\m}^{-}(\alpha) = \begin{cases} \left(\frac{12}{n_1}\right) \left(\frac{-4}{n_2} \right)(-1)^{\frac{n_2 -1}{2}} & n_1 \equiv n_2 \pmod{4} \\
 \left(\frac{12}{n_1}\right) \left(\frac{-4}{n_2} \right)(-1)^{\frac{n_1 -n_2 +2}{4}} & n_1 \equiv -n_2 \pmod{4}
\end{cases}
\end{equation*}
We let
\begin{equation*}
\phi_{14}(z; \tau) =\frac{1}{2} \left( \phi_{\xi^{+}}^{*}(z; \tau) - \phi_{\xi^{-}}^{*}(z; \tau) \right),
\end{equation*}
then we have the following.
\begin{Prop}
Assume the notation above.  Then
\begin{align*}
\phi_{14}(z; \tau) &= \sum_{n_1, n_2 \in \Z} \left(\frac{4}{n_1 + n_2} \right) \left(\frac{n_1}{3} \right) \left( (-1)^{\frac{n_1 - n_2 -1}{2}} - (-1)^{\frac{n_1 +n_2 -1}{2}} \right) \zeta_{1}^{\frac{n_1}{12}} \zeta_{2}^{\frac{n_2}{4}} q^{\frac{n_{1}^2 + 3 n_{2}^2}{12}} \\
&+\sum_{\substack{n_1, n_2 \in \Z \\ n_1 \equiv n_2 \pmod{4}}} \left(\frac{12}{n_1} \right) \left(\frac{-4}{n_2} \right) \left( (-1)^{\frac{n_1 + n_2 -2}{4}} - (-1)^{\frac{n_2 -1}{2}} \right) \zeta_{1}^{\frac{n_1}{24}} \zeta_{2}^{\frac{n_2}{8}} q^{\frac{n_{1}^2 + 3 n_{2}^2}{48}} \\
&+\sum_{\substack{n_1, n_2 \in \Z \\ n_1 \equiv -n_2 \pmod{4}}} \left(\frac{12}{n_1} \right) \left(\frac{-4}{n_2} \right) \left( (-1)^{\frac{n_2 -1}{2}} - (-1)^{\frac{n_1 -n_2 +2}{4}} \right) \zeta_{1}^{\frac{n_1}{24}} \zeta_{2}^{\frac{n_2}{8}} q^{\frac{n_{1}^2 + 3 n_{2}^2}{48}} \\
&= \frac{\theta \left(\frac{z_1 + 3z_2}{24} \right) \theta \left(\frac{z_1 -3z_2}{24} \right) \theta \left( \frac{z_1}{12} \right) \theta \left( \frac{z_1 - z_2}{8} \right) \theta \left( \frac{z_1 + z_2}{8} \right) \theta \left( \frac{z_2}{4} \right)}{\eta(\tau)^4}.
\end{align*}
\end{Prop}
\begin{rmk}
This identity is equivalent to the MacDonald identity for the root system $G_2$.  To see this more clearly, let $r_1 = \frac{z_1 + 3z_2}{24}$ and $r_2 = \frac{z_1 -3z_2}{24}$.  Then the arguments of the theta functions are $r_1, r_2, r_1 +r_2, r_1 +2r_2, 2r_1 +r_2, r_1 -r_2$ which is a choice of positive roots for $G_2$.
\end{rmk}
Define the coefficients $\tau_{14}(n)$ by 
\begin{equation*}
\sum_{n \equiv 7 \pmod{12}} \tau_{14}(n) q^{\frac{n}{12}} = q^{\frac{7}{12}} \prod_{n \geq 1} (1-q^n)^{14}.
\end{equation*}

\begin{Cor}
Assume the notation above.  Then
\begin{align*}
&\tau_{14}(n) = -\frac{1}{1440 \sqrt{-3}} \sum_{\substack{n_1, n_2 \in \Z \\ n_{1}^2 + 3n_{2}^2 = n}} \left( \frac{4}{n_1 + n_2} \right) \left( \frac{n_1}{3} \right) \left( (-1)^{\frac{n_1 - n_2 -1}{2}} - (-1)^{\frac{n_1 + n_2 -1}{2}} \right)(n_1 - n_2 \sqrt{-3})^6 \\
&=\frac{1}{60}\sum_{\substack{n_1 \geq 0 \\ n_{1}^2 + 3n_{2}^2 =n}} (-1)^{\frac{n_1 -n_2 -1}{2}} \left( \frac{4}{n_2 (n_1 + n_2)} \right) \left( \frac{n_1}{3} \right)n_1 n_2 (n_1 +n_2)(n_1 -n_2)(n_1 +3n_2)(n_1 -3n_2) \\
&= -\frac{1}{184320 \sqrt{-3}} \left[\sum_{\substack{n_1 \equiv n_2 \pmod{4} \\ n_{1}^2 + 3n_{2}^2 = 4n}} \left(\frac{12}{n_1} \right) \left(\frac{-4}{n_2} \right) \left( (-1)^{\frac{n_1 + n_2 -2}{4}} - (-1)^{\frac{n_2 -1}{2}} \right) (n_1 - n_2 \sqrt{-3})^6 \right. \\
&\left. +\sum_{\substack{n_1, n_2 \in \Z \\ n_1 \equiv -n_2 \pmod{4}}} \left(\frac{12}{n_1} \right) \left(\frac{-4}{n_2} \right) \left( (-1)^{\frac{n_2 -1}{2}} - (-1)^{\frac{n_1 -n_2 +2}{4}} \right) (n_1 - n_2 \sqrt{-3})^6 \right] \\
&=\frac{1}{3840} \sum_{\substack{n_1 \equiv n_2 \pmod{4} \\ n_{1}^2 + 3n_{2}^2 = 4n}} \left( \frac{12}{n_1} \right) \left( \frac{4}{n_2} \right)n_1 n_2 (n_1 +n_2)(n_1 -n_2)(n_1 +3n_2)(n_1 -3n_2)
\end{align*}

\end{Cor}

\subsubsection{$r=26$}

For $r=26$ we will consider a linear combination of $4$ Hecke characters.  Two of the characters are from $K'=\Q(\sqrt{-3})$ and two are from $K'' =\Q(i)$.  For $K'$ we consider the same two characters of conductor $\m'=4 \sqrt{-3}$ from the $r=14$ case and denote these by $\xi_{\m'}^{\pm}$.  For $K''$ we consider the two characters of conductor $\m''=6$ from the $r=10$ case cubed and denote them by $\xi_{\m''}^{\pm} = \left(\xi_{6}^{\pm} \right)^3$.  Let
\begin{equation*}
\phi_{26}(z; \tau) = \frac{1}{4} \left(\phi_{\xi_{\m'}^{+}}^{*}(z; \tau) + \phi_{\xi_{\m'}^{-}}^{*}(z; \tau) -\phi_{\xi_{\m''}^{+}}^{*}(z; \tau) -\phi_{\xi_{\m''}^{-}}^{*}(z; \tau) \right).
\end{equation*}
We now restate Theorem \ref{26Thm}.
\begin{Prop} \label{26}
Assume the notation above.  Then
\begin{align*}
\phi_{26}(z; \tau) &= \frac{1}{2} \sum_{n_3, n_4 \in \Z} (-1)^{\frac{n_3 + n_4 -1}{2}}  \left( \frac{4}{n_3(n_4 +1)} \right) \left(\frac{n_3}{3} \right) \zeta_{3}^{\frac{n_3}{12}} \zeta_{4}^{\frac{n_4}{4}} q^{\frac{n_{3}^2 + 3n_{4}^2}{12}} \\
&+\frac{1}{4} \sum_{n_{3} \equiv \pm n_4 \pmod{8}} \left(\frac{12}{n_3} \right)  \zeta_{3}^{\frac{n_3}{24}} \zeta_{4}^{\frac{n_4}{8}} q^{\frac{n_{3}^2 + 3n_{4}^2}{48}} \\
&-\frac{1}{2} \sum_{(n_1, n_2) \equiv (\pm 1, 0) \pmod{3}} (-1)^{n_2} \left( \frac{4}{n_1 +n_2} \right) \zeta_{1}^{\frac{n_{1}}{6}} \zeta_{2}^{\frac{n_{2}}{6}} q^{\frac{n_{1}^2 + n_{2}^2}{12}} \\
&+\frac{1}{2} \sum_{(n_1, n_2) \equiv (0, \pm 1) \pmod{3}} (-1)^{n_2} \left( \frac{4}{n_1 +n_2} \right) \zeta_{1}^{\frac{n_{1}}{6}} \zeta_{2}^{\frac{n_{2}}{6}} q^{\frac{n_{1}^2 + n_{2}^2}{12}} \\
&=\frac{1}{2} \theta^{*} \left(\frac{z_3}{6}; 2 \tau \right) \theta_{4} \left( \frac{z_4}{2}; 2 \tau \right) + \frac{1}{2} \theta^{*} \left(\frac{z_3}{12}; \frac{\tau}{2} \right) \theta_{2} \left( \frac{z_4}{4}; \frac{\tau}{2} \right) \\
&+ \frac{1}{4} \left[ \theta^{*}\left( \frac{z_3}{12} +\frac{1}{4}; \frac{\tau}{2} \right) \theta_{2} \left(\frac{z_4}{4} -\frac{1}{4}; \frac{\tau}{2} \right) + \theta_{2}^{*} \left(\frac{z_3}{12} +\frac{1}{4}; \frac{\tau}{2} \right) \theta \left(\frac{z_4}{4} -\frac{1}{4}; \frac{\tau}{2} \right) \right] \\
&+ \frac{1}{4} \left[ \theta^{*}\left( \frac{z_3}{12} +\frac{1}{4}; \frac{\tau}{2} \right) \theta_{2} \left(\frac{z_4}{4} +\frac{1}{4}; \frac{\tau}{2} \right) - \theta_{2}^{*} \left(\frac{z_3}{12} +\frac{1}{4}; \frac{\tau}{2} \right) \theta \left(\frac{z_4}{4} +\frac{1}{4}; \frac{\tau}{2} \right) \right] \\
&-\frac{1}{2} \theta^{*} \left(\frac{z_1 + z_2}{6}; \tau \right) \theta^{*} \left(\frac{z_1 -z_2}{6}; \tau \right).
\end{align*}
\end{Prop}
\begin{proof}
The first equality for $\phi_{26}(z;\tau)$ follows from the Hecke characters constructed for the $r=10$ and $r=14$ cases.  It is then straightforward to show
\begin{equation*}
\sum_{n_3, n_4 \in \Z} (-1)^{\frac{n_3 + n_4 -1}{2}}  \left( \frac{4}{n_3(n_4 +1)} \right) \left(\frac{n_3}{3} \right) \zeta_{3}^{\frac{n_3}{12}} \zeta_{4}^{\frac{n_4}{4}} q^{\frac{n_{3}^2 + 3n_{4}^2}{12}} =\theta^{*} \left(\frac{z_3}{6}; 2 \tau \right) \theta_{4} \left( \frac{z_4}{2}; 2 \tau \right) 
\end{equation*}
and 
\begin{align*}
&\sum_{(n_1, n_2) \equiv (\pm 1, 0) \pmod{3}} (-1)^{n_2} \left( \frac{4}{n_1 +n_2} \right) \zeta_{1}^{\frac{n_{1}}{6}} \zeta_{2}^{\frac{n_{2}}{6}} q^{\frac{n_{1}^2 + n_{2}^2}{12}} \\
&-\sum_{(n_1, n_2) \equiv (0, \pm 1) \pmod{3}} (-1)^{n_2} \left( \frac{4}{n_1 +n_2} \right) \zeta_{1}^{\frac{n_{1}}{6}} \zeta_{2}^{\frac{n_{2}}{6}} q^{\frac{n_{1}^2 + n_{2}^2}{12}} \\
&= \theta^{*} \left(\frac{z_1 + z_2}{6}; \tau \right) \theta^{*} \left(\frac{z_1 -z_2}{6}; \tau \right)
\end{align*}
by simply manipulating the Jacobi symbols inside the sums.  We then need to evaluate 
\begin{equation*}
\sum_{n_{3} \equiv \pm n_4 \pmod{8}} \left(\frac{12}{n_3} \right) \left(\frac{4}{n_4} \right) \zeta_{3}^{\frac{n_3}{24}} \zeta_{4}^{\frac{n_4}{8}} q^{\frac{n_{3}^2 + 3n_{4}^2}{48}}.
\end{equation*}
This is done by inserting
\begin{equation*}
\frac{1}{4} \left(1 +(-1)^{\frac{n_3 -n_4}{2}} \right) \left( 1 + (-1)^{\frac{n_3 -n_4}{4}} \right) = \begin{cases} 1 & n_3 \equiv n_4 \pmod{8} \\ 0 & n_3 \not\equiv n_4 \pmod{8}
\end{cases}
\end{equation*}
and
\begin{equation*}
\frac{1}{4} \left(1 +(-1)^{\frac{n_3 +n_4}{2}} \right) \left( 1 + (-1)^{\frac{n_3 +n_4}{4}} \right) = \begin{cases} 1 & n_3 \equiv -n_4 \pmod{8} \\ 0 & n_3 \not\equiv -n_4 \pmod{8}
\end{cases}
\end{equation*}
for $n_3$ and $n_4$ both odd into the relevant sums.  The theta function identities are then obtained by naively multiplying these terms out and manipulating the Jacobi symbols inside each sum.  The shifts in the theta functions by $\pm \frac{1}{4}$ arise from the terms $(-1)^{\frac{n_3-n_4}{4}}, (-1)^{\frac{3(n_3-n_4)}{4}}, (-1)^{\frac{n_3+n_4}{4}}$, and $(-1)^{\frac{3(n_3+n_4)}{4}}$.  For example,
\begin{equation*}
\theta^{*}\left( \frac{z_3}{12} +\frac{1}{4}; \frac{\tau}{2} \right) \theta_{2} \left(\frac{z_4}{4} +\frac{1}{4}; \frac{\tau}{2} \right) = \sum_{n_3, n_4 \in \Z} \left(\frac{12}{n_3} \right) \left( \frac{4}{n_4} \right) (-1)^{\frac{n_3 +n_4}{4}} \zeta_{3}^{\frac{n_3}{24}} \zeta_{4}^{\frac{n_4}{8}} q^{\frac{n_{3}^2 +3n_{4}^2}{48}}.
\end{equation*}

\end{proof}

\subsection{Theta blocks for elliptic curves with CM}
In \cite{MO} Martin and Ono classify all of the weight $2$ modular forms associated to elliptic curves which are eta-quotients.  Of those eta-quotients they also determined which corresponded to elliptic curves with complex multiplication.  These forms are tabulated below.
\begin{center}
\begin{table}[H]
\begin{tabular}{ | m{3.25cm} | m{3.25cm} | m{3.25cm} | m{3.25cm} |}
\hline
 Conductor & Eta-quotient & Field & Elliptic curve  \\ 
 \hline
 $27$ & $\eta^{2}(3 \tau) \eta^{2}(9\tau)$ & $\Q(\sqrt{-3})$ & $y^2 +y = x^3 -7$  \\
 \hline
 $32$ & $\eta^2(4 \tau) \eta^2 (8 \tau)$ & $\Q(i)$ & $y^2 = x^3 +4x$  \\
 \hline   
 $36$ & $\eta^4(6 \tau)$ & $\Q(\sqrt{-3})$ & $y^2 =x^3 +1$  \\
 \hline
   $64$ & $\frac{\eta^8(8 \tau)}{\eta^{2}(4 \tau) \eta^{2}(16 \tau)}$ & $\Q(i)$ & $y^2 = x^3 -4x$  \\
 \hline
 $144$ & $\frac{\eta^{12}(12 \tau)}{\eta^{4}(6 \tau) \eta^{4}(24 \tau)}$ & $\Q(\sqrt{-3})$ & $y^2 =x^3 -1$ \\
 \hline
\end{tabular}
 \caption{Eta-quotients corresponding to elliptic curves with complex multiplication.}
 \end{table}
 \end{center}
Using this result we can classify all theta blocks associated to elliptic curves with complex multiplication.  We say a Jacobi form $\phi(z; \tau)$ corresponds to an elliptic curve if $D_{z}\left[\phi(z;\tau)\right]_{z=0}$ is equal to the weight $2$ newform associated to the elliptic curve, up to a nonzero constant.
\begin{Thm} \label{elliptic}
Let $\phi_{N}(z; \tau)$ correspond to the elliptic curve with Conductor $N$.  Then the theta blocks corresponding to elliptic curves with CM are
\begin{enumerate}
\item $N=27$: 
\begin{align*}
&\phi_{27}(z; 3\tau) = \sum_{n_1 \equiv n_2 \pmod{2}} \left( \frac{2n_1}{3} \right) \rho^{3n_1 -n_2} \zeta_{1}^{\frac{n_{1}}{6}} \zeta_{2}^{\frac{n_{2}}{2}} q^{\frac{3\left(n_{1}^2 +3n_{2}^2\right)}{4}} \\
&= \frac{\theta\left(\frac{z_1 +3z_2}{6} + \frac{1}{3};3 \tau \right) \theta \left( \frac{z_1 -3z_2}{6} +\frac{2}{3}; 3\tau \right) \theta \left(\frac{z_1}{3};3 \tau \right)}{\eta(3\tau)}.
\end{align*}
\item $N=32$:
\begin{align*}
\phi_{32}(z; \tau) &=-\sum_{n_1, n_2 \in Z} \left(\frac{-4}{n_1(n_2-1)} \right) \zeta_{1}^{\frac{n_1}{4}} \zeta_{2}^{\frac{n_2}{4}} q^{n_{1}^2 + n_{2}^2} \\
&= \theta \left(\frac{z_1}{2}; 8 \tau \right) \theta_{4} \left(\frac{z_2}{2}; 8 \tau \right).
\end{align*}
\item $N=36$:
\begin{align*}
\phi_{36}(z; \tau) &= \sum_{n_1, n_2 \in \Z} \left(\frac{12}{n_1} \right) \left(\frac{-4}{n_2} \right) \zeta_{1}^{\frac{n_1}{12}} \zeta_{2}^{\frac{n_2}{4}} q^{\frac{n_{1}^2 + 3 n_{2}^2}{4}} \\
&= \theta^{*} \left(\frac{z_1}{6}; 6\tau \right) \theta \left(\frac{z_2}{2}; 6 \tau \right).
\end{align*}
\item $N=64$:
\begin{align*}
\phi_{64}(z; \tau) &= \sum_{n_1, n_2 \in \Z} (-1)^{\frac{n_1 -1}{2}} \left( \frac{4}{n_1(n_2 -1)} \right) \zeta_{1}^{\frac{n_1}{4}} \zeta_{2}^{\frac{n_2}{4}} q^{n_{1}^2 + n_{2}^2} \\
&= \theta \left(\frac{z_1}{2}; 8 \tau \right) \theta_{3} \left(\frac{z_2}{2}; 8 \tau \right).
\end{align*}
\item $N=144$:
\begin{align*}
\phi_{144}(z; \tau) &= \frac{1}{2} \sum_{n_1, n_2 \in \Z} \left(1 + (-1)^{\frac{n_1 + n_2}{2}} \right) \left(\frac{12}{n_1 -3} \right) \left( \frac{4}{n_2 -1} \right) \zeta_{1}^{\frac{n_1}{12}} \zeta_{2}^{\frac{n_2}{4}} q^{\frac{n_{1}^2 + 3n_{2}^2}{4}} \\
&=\frac{1}{2} \left[ \theta_{3}^{*} \left( \frac{z_1}{6}; 6 \tau \right) \theta_{4} \left(\frac{z_2}{2}; 6 \tau \right) + \theta_{4}^{*} \left(\frac{z_1}{6}; 6 \tau \right) \theta_{3} \left( \frac{z_2}{2}; 6 \tau \right) \right].
\end{align*}
\end{enumerate}
\end{Thm}
\begin{rmk}
The identity for $N=27$ seems related to the cubic theta functions studied in \cite{Bor}.
\end{rmk}
\begin{rmk}
These Jacobi forms are of particular interest as in \cite{RV} Rodriguez Villegas showed the square of their $(k-1)$-st Taylor coefficient evaluated at a CM point is essentially the central value of $L(\xi^{2k-1}, k)$ where $\xi$ is the associated Hecke character.
\end{rmk}
\begin{proof}
In each case one can construct an explicit Hecke character to obtain the given Jacobi form.  However, here all of the identities except the $N=27$ and $N=144$ cases are evident from multiplying two Jacobi theta functions from Proposition \ref{BB}.  The $N=144$ case is a sum two products of the building block Jacobi theta functions and can be seen to be associated to the Elliptic curve by the identity
\begin{equation*}
\frac{\eta(12 \tau)^{12}}{\eta(6 \tau)^4 \eta(24 \tau)^4} = \frac{1}{2} \left[ \frac{\eta(6 \tau)^{10}}{\eta(3 \tau)^4 \eta(12 \tau)^2} + \frac{\eta(3 \tau)^4 \eta(12 \tau)^2}{\eta(6 \tau)^2} \right].
\end{equation*}
The $N=27$ case can be recognized as the $r=8$ form of the previous section with the shifts $z_1 \mapsto z_1 +3$ and $z_2 \mapsto z_2 -\frac{1}{3}$.
\end{proof}

\subsection{Ramanujan-type congruences and partition statistics}

We begin this subsection by showing how Ramanujan's partition congruences and the crank statistic connect to modular forms and Jacobi forms with CM and then use this case as a guideline.  Let $\ell$ be a prime and notice that modulo $\ell$ one can rewrite the partition generating function as 
\begin{equation*}
\prod_{n \geq 1} \frac{1}{1-q^n} = \prod_{n \geq 1} \frac{(1-q^n)^{\ell -1}}{(1-q^n)^{\ell}} \equiv \prod_{n \geq 1} \frac{(1-q^n)^{\ell -1}}{1-q^{\ell n}} \pmod{\ell}.
\end{equation*}
The denominator is only supported on exponents which are multiples of $\ell$ so one immediately finds that $p(\ell n + \delta) \equiv 0 \pmod{\ell}$ if and only if $a_{\ell -1}(\ell n + \delta) \equiv 0 \pmod{\ell}$ where $\sum_{n \geq 0} a_{k}(n) q^n := q^{-\frac{k}{24}} \eta(\tau)^{k}$.  One can then use the fact that $\eta^4, \eta^6$, and $\eta^{10}$ are modular with CM to quickly confirm the congruences.  Often when a Ramanujan-type congruence exists, it appears there is a theta function or modular form with CM in the background like this to explain it.  This idea was used by Boylan \cite{Boylan} and the author and Locus Dawsey \cite{DW} to prove Ramanujan-type congruences for $k$-colored partitions.  One can use a similar idea in order to discover the crank.  In general, if a function $f(\tau) = \sum_{n \geq 0} a(n) q^n$ satisfies a Ramanujan congruence $a(\ell n + \delta) \equiv 0 \pmod{\ell}$, then a function $f(z; \tau) = \sum_{n \geq 0} a_{n}(\zeta)q^n$ defines a statistic that explains this congruence if and only if $f(0; \tau) = f(\tau)$ and $\Phi_{\ell}(\zeta) \big| \left[q^{\ell n + \delta} \right] f(z; \tau)$ as Laurent polynomials, where $\Phi_{\ell}$ is the $\ell$-th cyclotomic polynomial \cite{RTW}.  One can then use the Jacobi forms $\phi_{r}(z; \tau)$ from Section 4.1 with $r=4,6,10$ to find the crank function.  Specifically, we want to find $C(z; \tau)$ such that
\begin{equation*}
C(z; \tau) = \frac{\widetilde{\phi}_{r}(az, bz; \tau)}{G(z; \tau)},
\end{equation*}
where $\widetilde{\phi}_{r}(z; \tau)$ is $\phi_r (z; \tau)$ with prefactors removed and we have specialized $z_1 =az$ and $z_2 =bz$ for choices of $a$ and $b$ to be determined.  One can quickly check that $\Phi_{\ell}(\zeta)\big| \left[q^{\ell n + \delta} \right] \widetilde{\phi}_{r}(z; \tau)$ as the primes $\ell = 5, 7, 11$ are inert in the field associated to $\xi$ for $r=4, 6$, and $10$ respectively.  One now wants the denominator to satisfy something equivalent to being supported on multiples of $\ell$.  Let $(1-\zeta^{\pm a} q^n) = (1-\zeta^a q^n)(1-\zeta^{-a}q^n)$. Then we want the denominator to have the shape
\begin{equation*}
\prod_{n \geq 1} \left[ (1-q^n) (1-\zeta^{\pm a_1} q^n)(1-\zeta^{\pm a_2}q^n) \cdots (1-\zeta^{\pm a_{\frac{\ell -1}{2}}}q^n) \right]^j = \left[1-q^{\ell n} + \Phi_{\ell}(\zeta) g(z; \tau) \right]^j,
\end{equation*}
where $\pm(a_1, a_2, \dots, a_{\frac{\ell -1}{2}})$ is a complete residue system modulo $\ell$, because then every coefficient of $G(z; \tau)^{-1}$ not supported on a multiple of $\ell$ is clearly divisible by $\Phi_{\ell}(\zeta)$.  The specializations $z_1 =az$ and $z_2 =bz$ then should be chosen to guarantee the quotient of $\widetilde{\phi}_{r}$ and $G(z;\tau)$ simplifies in some way.  For example, when $\ell =7$ we have
\begin{align*}
\widetilde{\phi}_{6}(az,bz;\tau) &= \frac{\zeta^{\frac{a}{2}}}{(1-\zeta^{\frac{a+b}{2}})(1-\zeta^{\frac{a-b}{2}})} \sum_{n_1, n_2 \in \Z} (-1)^{n_2} \left( \frac{4}{n_{1}^2 + n_{2}^2} \right) \zeta^{an_1 +bn_2} q^{\frac{n_{1}^2 + n_{2}^2 -1}{4}} \\
&= q^{-\frac{1}{4}} \frac{\zeta^{\frac{a}{2}}}{(1-\zeta^{\frac{a+b}{2}})(1-\zeta^{\frac{a-b}{2}})} \theta \left( \frac{a+b}{2} z \right) \theta \left( \frac{a-b}{2} z \right) \\
&= \prod_{n \geq 1} (1-q^n)^2 (1-\zeta^{\pm(a+b)/2} q^n)(1-\zeta^{\pm(a-b)/2}q^n).
\end{align*}
Then if we choose $a=5$ and $b=1$, our crank is
\begin{align*}
C(z; \tau) &= \prod_{ n \geq 1} \frac{(1-q^n)^2 (1-\zeta^{\pm 2} q^n)(1-\zeta^{\pm 3}q^n)}{(1-q^n)(1-\zeta^{\pm 1}q^n)(1-\zeta^{\pm 2} q^n)(1-\zeta^{\pm 3}q^n)} \\
&= \prod_{n \geq 1} \frac{1-q^n}{(1-\zeta^{1}q^n)(1-\zeta^{-1} q^n)}.
\end{align*}
Note that this choice of crank function is not unique and this discussion only shows it's a crank for $\ell=7$.  However, this is the crank discovered by Andrews and Garvan and it can similarly be shown to simultaneously explain the congruences for $\ell=5$ and $11$.  The known formulas for the Jacobi forms with CM associated to powers of $\eta$ and this technique was used by Rolen, Tripp, and the author in \cite{RTW} to define crank functions for $k$-colored partitions.  The point of this section is to show that cranks can be found for other partition functions using the theory of Jacobi forms with CM.

We will now discuss a doubly infinite family of partition functions with Ramanujan-type congruences and show how to construct a crank function in any specific case.  An {\it overpartition} of $n$ is a partition of $n$ where the first occurrence of any number may be overlined.  As the overlined parts form a partition into distinct parts it's easy that the overpartition generating function is
\begin{equation*}
\sum_{n \geq 0} \overline{p}(n) q^n \coloneqq \prod_{n \geq 1} \frac{1+q^n}{1-q^n}.
\end{equation*} 
We give the following generalization.
\begin{defn}
For $0 < j \leq k$, a {\bf{$(k,j)$-colored overpartition}} of $n$ is a $k$-colored partition of $n$ where the first occurrence of any number of $j$ of the colors may be overlined.  If $\overline{p}_{k,j}(n)$ denotes the number of $(k,j)$-colored overpartitions of $n$, then we have the generating function
\begin{equation*}
\sum_{n \geq 0} \overline{p}_{k,j}(n) q^n := \prod_{n \geq 1} \frac{(1+q^n)^j}{(1-q^n)^k}, \qquad 0 < j \leq k.
\end{equation*}
\end{defn}
We now give the first examples of Ramanujan-type congruences.  This result is analogous to Theorem 2 of \cite{KO} for $k$-colored partitions.
\begin{Thm} \label{Simple}
We have the Ramanujan-type congruence
\begin{equation*}
\overline{p}_{k,j}(\ell n + \delta) \equiv 0 \pmod{\ell}
\end{equation*}
if any of the following hold:
\begin{enumerate}
\item $(k+j, -j) \equiv (-1,0) \pmod{\ell}$ and $\left(\frac{24 \delta +1}{\ell} \right)=-1$ for $\ell \geq 5$ prime.
\item $(k+j, -j) \equiv (0,-1) \pmod{\ell}$ and $\left(\frac{12 \delta +1}{\ell} \right)=-1$ for $\ell \geq 5$ prime.
\item $(k+j, -j) \equiv (-3,0) \pmod{\ell}$ and $\left(\frac{8 \delta +1}{\ell} \right)=-1$ or $8 \delta +1 \equiv 0 \pmod{\ell}$ for $\ell \geq 5$ prime.
\item $(k+j, -j) \equiv (0,-3) \pmod{\ell}$ and $\left(\frac{4 \delta +1}{\ell} \right)=-1$ or $4 \delta +1 \equiv 0 \pmod{\ell}$ for $\ell \geq 5$ prime.
\item $(k+j, -j) \equiv (-2,1) \pmod{\ell}$ and $\left(\frac{\delta}{\ell} \right)=-1$ for $\ell \geq 3$ prime.
\item If $(k+j, -j) \equiv (1,-2) \pmod{\ell}$ and $\left(\frac{8\delta+1}{\ell} \right)=-1$ for $\ell \geq 3$ prime.
\item $(k+j, -j) \equiv (-5,2) \pmod{\ell}$ and $\left(\frac{24\delta +1}{\ell} \right)=-1$ or $24 \delta +1 \equiv 0 \pmod{\ell}$ for $\ell \geq 5$ prime.
\item $(k+j, -j) \equiv (2,-5) \pmod{\ell}$ and $\left(\frac{3\delta +1}{\ell} \right)=-1$ or $3 \delta +1 \equiv 0 \pmod{\ell}$ for $\ell \geq 5$ prime.
\end{enumerate}
\end{Thm}
\begin{proof}
These simple congruences for the $(k,j)$-colored overpartitions come from the eta-quotients $\eta(\tau), \eta(2\tau), \eta(\tau)^3, \eta(2 \tau)^3, \frac{\eta(\tau)^2}{\eta(2 \tau)}, \frac{\eta(2 \tau)^2}{\eta(\tau)}, \frac{\eta(\tau)^5}{\eta(2 \tau)^2}$, and $\frac{\eta(2 \tau)^{5}}{\eta(\tau)^2}$ of weights $\frac{1}{2}$ and $\frac{3}{2}$ which are unary theta functions. In particular, the result follows from the cases where $\sum_{n \geq 0} \overline{p}_{k,j}(n) q^n$ is congruent, up to a power of $q$, to one of these forms modulo $\ell$ and from the well-known $q$-expansions \cite{LO}:
\begin{align*}
\eta(\tau) &= \sum_{n \in \Z} \left(\frac{12}{n} \right) q^{\frac{n^2}{24}} \\
\eta(\tau)^3 &= \sum_{n \geq 1} \left(\frac{-4}{n} \right) n q^{\frac{n^2}{8}} \\
\frac{\eta(\tau)^2}{\eta(2 \tau)} &= \sum_{n \in \Z} (-1)^n q^{n^2} \\
\frac{\eta(2 \tau)^2}{\eta(\tau)} &= \sum_{n \geq 1} \left(\frac{n}{4} \right) q^{\frac{n^2}{8}} \\
\frac{\eta(\tau)^5}{\eta(2 \tau)^2} &= \sum_{n \geq 1} \left(\frac{n}{12} \right) n q^{\frac{n^2}{24}} \\
\frac{\eta(2 \tau)^5}{\eta(\tau)^2} &= \sum_{n \geq 1} \left( 2 \left(\frac{n}{12} \right) - \left(\frac{n}{3} \right)\right) n q^{\frac{n^2}{3}}.
\end{align*}
\end{proof}
\begin{Ex}
In the case of $j=k$, $\overline{p}_{k,j}(n) = \overline{p}_{k}(n)$ counts the number of $k$-colored overpartitions.  Let $\ell$ be an odd prime. In this case, when $k= \ell t -1$ we have
\begin{equation*}
\sum_{n \geq 0} \overline{p}_{k}(n) q^n \equiv \prod_{n \geq 1} \frac{(1+q^{\ell t n})}{(1-q^{\ell t n})} \cdot \frac{(1-q^n)^2}{(1-q^{2n})} \pmod{\ell}
\end{equation*}
and therefore
\begin{equation*}
\overline{p}_{\ell t -1}(\ell n + \delta) \equiv 0 \pmod{\ell}, \qquad \left(\frac{\delta}{\ell} \right)=-1.
\end{equation*}
\end{Ex}

Serre's \cite{Serre} classification of even powers of $\eta$ with CM allows one to prove more interesting Ramanujan-type congruences for $k$-colored partitions.  Gordon and Robins \cite{GR} proved the analogous result of Serre for eta-quotients of the form $\eta(\tau)^r \eta(2 \tau)^{s}$ with $r+s$ even.
\newpage
\begin{thm*}[Serre, Gordon-Robins]
Suppose that $r+s$ is even.  Then $\eta(\tau)^r \eta(2 \tau)^s$ has complex multiplication by $\Q(\sqrt{D})$ if and only if $(r,s)$ is one of the pairs:
\begin{center}
\begin{table}[H]
\begin{tabular}{ | m{1cm} | m{3.25cm} | m{3.25cm} | m{3.25cm} | m{3.25cm} |}
\hline
 $(r,s)$ & $\Q(i)$ & $\Q(\sqrt{-2})$ & $\Q(\sqrt{-3})$ & $\Q(\sqrt{-6})$  \\ 
 \hline
 $k=1$ & $(1,1), (3,-1)$, $(-1,3), (4,-2)$, $(-2,4), (2,0)$, $(0,2)$ &  &  & \\
 \hline
 $k=2$ & $(2,2), (6,-2)$, $(-2,6)$ & $(5,-1), (-1,5)$ & $(4,0), (0,4)$ & $(3,1), (1,3)$, $(7,-3), (-3,7)$  \\
 \hline   
 $k=3$ & $(4,2), (2,4)$, $(7,-1), (-1,7)$, $(9,-3), (-3,9)$, $(10,-4), (-4,10)$, $(11,-5), (-5,11)$, $(6,0), (0,6)$ & $(3,3)$ & $(8,-2), (-2,8)$ & $(5,1), (1,5)$  \\
 \hline
   $k=4$ &  &  & $(8,0), (0,8)$ &  \\
 \hline
 $k=5$ & $(5,5), (7,3)$, $(3,7), (14,-4)$, $(-4,14), (16,-6)$, $(-6,16), (17,-7)$, $(-7,17), (18,-8)$, $(-8,18), (15,-5)$, $(-5,15), (19,-9)$, $(-9,19), (10,0)$, $(0,10)$ & $(17,-7), (-7,17)$ & $(16,-6), (-6,16)$, $(18,-8), (-8,18)$ & $(15,-5), (-5,15)$, $(19,-9), (-9,19)$ \\
 \hline
    $k=4$ &  &  & $(14,0), (0,14)$ &  \\
 \hline
  $k=9$ & $(9,9)$ & $(9,9)$ &  &  \\
 \hline
    $k=13$ & $(26,0), (0,26)$ &  & $(26,0), (0,26)$ &  \\
 \hline
\end{tabular}
 \caption{Pairs $(r,s)$ such that $\eta(\tau)^r \eta(2 \tau)^s$ has weight $k=\frac{r+s}{2}$ and CM by $\Q(\sqrt{D})$.}
 \end{table}
 \end{center}
 \end{thm*}
 \begin{rmk}
 If the same pair $(r,s)$ is in two different columns in Table 1, then it means the form $\eta(\tau)^r \eta(2 \tau)^s$ is a linear combination of CM forms with Hecke characters coming from the associated fields.
 \end{rmk}
 Using this result we are able to prove further Ramanujan-type congruences for $(k,j)$-colored overpartitions.
 \begin{Thm} \label{Cong}
 Let $\ell \geq 5$ be prime and $\delta_{k,j,\ell}$ be such that $24 \delta_{k,j,\ell} \equiv k-j \pmod{\ell}$.  Then we have the Ramanujan-type congruence
 \begin{equation*}
 \overline{p}_{k,j}(\ell n + \delta_{k,j,l}) \equiv 0 \pmod{\ell}
 \end{equation*}
 if $(k+j, -j) \equiv -(r,s) \pmod{\ell}$ for any of the pairs in Table 1 with $\frac{r+s}{2} \neq 1$ and $\ell$ is inert in the associated $\Q(\sqrt{D})$.
 \end{Thm}
 \begin{rmk}
 If a pair $(r,s)$ shows up for two different fields in Table 1, then $\ell$ must be inert in both of those fields.
 \end{rmk}
\begin{proof}
If $(k+j, -j) \equiv -(r,s) \pmod{\ell}$ then $\sum_{n \geq 0} \overline{p}_{k,j}(n) q^n$ is, up to powers of $q$, congruent modulo $\ell$ to $\eta(\tau)^r \eta(2 \tau)^{s}$ times a $q$-series only supported on exponents which are multiples of $\ell$.  The congruence is then equivalent to the same congruence for the eta-quotient with CM, where the missing $q$-power is accounted for by the shift $\delta_{k,j,\ell}$.  The congruence for the CM form then follows as $\ell$ is inert in the associated field or fields.
\end{proof}
\begin{Ex}
Consider $3$-colored partitions where the first occurrence of any number in $2$ of the colors may be overlined.  These $(3,2)$-colored overpartitions have generating function given by
\begin{equation*}
\sum_{n \geq 0} \overline{p}_{3,2}(n) q^n = \prod_{n \geq 1} \frac{(1+q^n)^{2}}{(1-q^n)^3}.
\end{equation*}
We have $(k+j,-j) = (5,-2) \equiv -(2,2) \pmod{7}$, and $(5,-2) \equiv -(-5,15) \pmod{13}$ and $7$ and $13$ are inert in $\Q(i)$ and $\Q(\sqrt{-6})$ respectively.  However, $13$ is not inert in $\Q(i)$ so we only obtain a congruence modulo $7$.  We find
\begin{equation*}
\overline{p}_{3,2}(7n+5) \equiv 0 \pmod{7}.
\end{equation*}
\end{Ex}
If one now wants to construct crank functions to explain these congruences they just need to look at the Jacobi form with CM that specializes to the modular form with CM that gave the congruence.  We begin by looking at $2$-colored overpartitions.  These were studied by Bringmann and Lovejoy in \cite{BL} where they were called overpartition pairs.  It was proved in \cite{BL} that
\begin{equation*}
\overline{p}_{2}(3n+2) \equiv 0 \pmod{3}.
\end{equation*}
One can also see this from Theorem \ref{Simple} as $(k+j, -j) = (4,-2) \equiv (1,-2) \pmod{3}$.  The relevant CM form here is 
\begin{equation*}
\frac{\eta(2 \tau)^2}{\eta(\tau)} = \sum_{n \geq 1} \left( \frac{n}{4} \right) q^{\frac{n^2}{8}}.
\end{equation*}
The relevant Jacobi form living above this modular form is
\begin{align*}
\theta_{2}(z; \tau) &= \sum_{n \in \Z} \left(\frac{4}{n} \right) \zeta^{\frac{n}{2}} q^{\frac{n^2}{8}} \\
&=\frac{\eta(\tau)^2 \theta(2z; 2\tau)}{\eta(2 \tau) \theta(z; \tau)} = q^{\frac{1}{8}} \left( \zeta^{\frac{1}{2}} + \zeta^{-\frac{1}{2}} \right) \prod_{n \geq 1} \frac{(1-q^n)(1-\zeta^2 q^{2n})(1-\zeta^{-2}q^{2n})}{(1-\zeta q^n)(1-\zeta^{-1}q^n)} \\
&= q^{\frac{1}{8}} \left( \zeta^{\frac{1}{2}} + \zeta^{-\frac{1}{2}} \right) \prod_{n \geq 1} (1-q^n)(1+\zeta q^n)(1+\zeta^{-1}q^n).
\end{align*}
We then claim a crank function for the $2$-colored overpartitions is
\begin{equation*}
\overline{C}_{2}(z; \tau) := \prod_{n \geq 1} \frac{(1+\zeta q^n)(1+\zeta^{-1}q^n)}{(1-\zeta q^n)(1-\zeta^{-1}q^n)}.
\end{equation*}
To see this note that 
\begin{align*}
\overline{C}_{2}(z; \tau) &= \tilde{\theta}_{2}(z;\tau) \prod_{n \geq 1}\frac{1}{\left[ 1-q^{3n} + \Phi_{3}(\zeta)g_{n}(z;\tau) \right]}
\end{align*}
and so $\Phi_{3}(\zeta) \big| \left[q^{3n + 2} \right] \overline{C}_{2}(z; \tau)$.
\begin{rmk}
The crank function for the $2$-colored overpartitions studied in \cite{BL} is slightly different than the one presented here.
\end{rmk}
We now show how to construct a crank function for the more interesting congruence $\overline{p}_{3,2}(7n+5) \equiv 0 \pmod{7}$.  Here the relevant CM form is
\begin{equation*}
\eta(\tau)^2 \eta(2 \tau)^2 = -\frac{1}{2} \sum_{n_1, n_2 \in \Z} \left(\frac{-4}{n_1 (n_2 -1)} \right)(n_1 +n_2) q^{\frac{n_{1}^2 + n_{2}^2}{4}}.
\end{equation*}
This form has complex multiplication with respect to $\Q(i)$ with a character of conductor $\m =(1+i)^3$.  Therefore $\alpha = n_1 + i n_2$ is coprime to $\m$ if $n_1 \not\equiv n_2 \pmod{2}$ and $\alpha \equiv^{*} 1 \pmod{\m}$ if $n_1 \equiv 1 \pmod{4}$ and $n_2 \equiv 0 \pmod{4}$.  We consider the character 
\begin{equation*}
\xi_{(1+i)^3}(\alpha) =\begin{cases} 1 & \text{if} \ n_1 + n_2 \equiv 1 \pmod{4} \\ -1 & \text{if} \ n_1 + n_2 \equiv -1 \pmod{4}. \end{cases}
\end{equation*}
Explicitly we have 
\begin{equation*}
\xi_{(1+i)^3}(\alpha) =\begin{cases} -\left( \frac{-4}{n_1 (n_2 -1)} \right) & \text{if} \ n_1  \equiv 1 \pmod{2} \\ -\left(\frac{-4}{(n_1 -1)n_2} \right) & \text{if} \ n_1  \equiv 0 \pmod{2}. \end{cases}
\end{equation*}
We therefore have the following.
\begin{Prop}
Assume the notation above.  Then
\begin{align*}
\phi_{\xi}^{*}(z; \tau) &= -\sum_{n_1, n_2 \in \Z} \left[ \left(\frac{-4}{n_1 (n_2 -1)} \right) + \left(\frac{-4}{(n_1 -1)n_2} \right) \right] \zeta_{1}^{\frac{n_1}{2}} \zeta_{2}^{\frac{n_2}{2}} q^{\frac{n_{1}^2 + n_{2}^2}{4}} \\
&= \theta(z_1 ; 2 \tau) \theta_{4}(z_2; 2\tau) + \theta(z_2; 2\tau) \theta_{4}(z_1; 2 \tau).
\end{align*}
\end{Prop}
For the purposes of constructing a crank function we can of course instead take
\begin{align*}
\psi_{\xi}(z; \tau) &= -\sum_{n_1, n_2 \in \Z} \left(\frac{-4}{n_1 (n_2 -1)} \right) \zeta_{1}^{\frac{n_1}{2}} \zeta_{2}^{\frac{n_2}{2}} q^{\frac{n_{1}^2 + n_{2}^2}{4}} \\
&= \theta(z_1 ; 2 \tau) \theta_{4}(z_2; 2\tau) \\
&= q^{\frac{1}{4}}(\zeta_{1}^{\frac{1}{2}} - \zeta_{1}^{-\frac{1}{2}}) \prod_{n \geq 1} \frac{(1-q^{2n})^2 (1-\zeta_{1}^{\pm 1} q^{2n})(1-\zeta_{2}^{\pm 1} q^n)}{(1- \zeta_{2}^{\pm 1}q^{2n})}.
\end{align*}
In fact, we will only need 
\begin{align*}
\psi_{\xi}(3z; \tau) &=\theta(3z;2 \tau) \theta_{4}(3z; 2 \tau) = -\sum_{n_1, n_2 \in \Z} \left(\frac{-4}{n_1 (n_2 -1)} \right) \zeta^{\frac{3}{2}(n_1 +n_2)} q^{\frac{n_{1}^2 + n_{2}^2}{4}} \\
&=q^{\frac{1}{4}}(\zeta^{\frac{3}{2}} - \zeta^{-\frac{3}{2}}) \prod_{n \geq 1} (1-q^{2n})^2 (1-\zeta^{\pm 3} q^n) \\
&= \frac{\eta(2 \tau)^2}{\eta(\tau)} \theta(3z; \tau).
\end{align*}
for the following.
\begin{Thm}
Let 
\begin{equation*}
\overline{C}_{3,2}(z; \tau) := \prod_{n \geq 1} \frac{(1-q^n)(1+q^n)^2}{(1-\zeta^{\pm1} q^n)(1-\zeta^{\pm2}q^n)}.
\end{equation*}
Then $\overline{C}_{3,2}(z; \tau)$ defines a crank function for the $(3,2)$-colored overpartitions.  In particular,
\begin{equation*}
\Phi_{7}(\zeta) \big| \left[q^{7n+5}\right] \overline{C}_{3,2}(z; \tau).
\end{equation*}
\end{Thm}
\begin{proof}
We multiply $\overline{C}_{3,2}(z; \tau)$ by $1=\prod_{n \geq 1} \frac{(1-q^n)(1- \zeta^{\pm3} q^n)}{(1-q^n)(1- \zeta^{\pm3} q^n)}$ to obtain
\begin{align*}
\overline{C}_{3,2}(z; \tau) &= \prod_{n \geq 1} \frac{(1-q^{2n})^2 (1-\zeta^{\pm 3}q^n)}{(1-q^n)(1-\zeta^{\pm1} q^n)(1-\zeta^{\pm2} q^n)(1-\zeta^{\pm3} q^n)} \\
&= \tilde{\psi}_{\xi}(3z; \tau) \cdot \prod_{n \geq 1} \frac{1}{1-q^{7n} + \Phi_{7}(\zeta)g_{n}(z; \tau)}.
\end{align*}
The theorem then follows from the expression given for $\widetilde{\psi}_{\xi}(3z; \tau)$.
\end{proof}

\end{document}